\documentclass[preprint]{elsarticle}

\usepackage{hyperref}

\journal{Journal of \LaTeX\ Templates}

\usepackage{bm}
\usepackage{amsmath, amssymb, amsfonts,amsthm}
\usepackage{array}
\usepackage{epsfig}
\usepackage{enumerate}
\usepackage{bbm}
\usepackage{color}
\usepackage{float}
\usepackage[mathscr]{euscript}
\usepackage{mathtools}
\usepackage{graphicx,subcaption}
\usepackage{thmtools}
\usepackage{environ}
\usepackage{mathptmx}
\usepackage[draft]{todonotes} 
\usepackage{pgfplots}
\usetikzlibrary{matrix}

  \usetikzlibrary{positioning,calc,fit,shapes.geometric,patterns}
  \pgfdeclarelayer{background}
  \pgfdeclarelayer{foreground}
  \pgfsetlayers{background,main,foreground}
  \tikzstyle{vec}=[circle,inner sep=1pt,outer sep=-1pt,fill]
  \tikzstyle{border}=[thick]
  \tikzstyle{favborder}=[border,dotted]
  \tikzstyle{exclborder}=[border,dashed]
\usepackage{etex,etoolbox}
\usepackage{hyperref}

\hypersetup{colorlinks=true,linkcolor=blue,citecolor=magenta}

 \newcommand{\set}[2]{\left\{#1\colon#2\right\}}
\newcommand{\Reals}{\mathbb{R}} 
 
\newcommand{\N}{\mathbb{N}}

\newcommand{\pspace}{\varOmega}
\newcommand{\domain}{\mathcal{K}}
\newcommand{\bdomain}{\mathcal{C}}
\newcommand{\ydomain}{\mathcal{Y}}

\newcommand{\reals}{\Reals}
\newcommand{\gambles}{\mathcal{L}}
\newcommand{\posi}{\mathsf{posi}}

\newtheorem{theorem}{Theorem}
\newtheorem{definition}{Definition}
\newtheorem{proposition}{Proposition}
\newtheorem{example}{Example}

\newtheorem{lemma}{Lemma}

\biboptions{authoryear}

\begin{document}

\begin{frontmatter}

\title{Sum-Of-Squares for bounded rationality}
\tnotetext[mytitlenote]{Fully documented templates are available in the elsarticle package on \href{http://www.ctan.org/tex-archive/macros/latex/contrib/elsarticle}{CTAN}.}


\author{ Alessio Benavoli, Alessandro Facchini,  Dario Piga, Marco Zaffalon}
\address{Istituto Dalle Molle di Studi Sull'Intelligenza Artificiale (IDSIA), Lugano (Switzerland)}

\begin{abstract}
In the gambling foundation of probability theory, rationality requires that a subject should always (never) find desirable  all nonnegative (negative) gambles, because no matter the result of the experiment the subject never (always) decreases  her money.
Evaluating the nonnegativity of a gamble in infinite spaces is a difficult task. In fact, even if we restrict the gambles to be polynomials in $\reals^n$, the problem  of determining  nonnegativity is NP-hard. The aim of this paper is to develop a \textit{computable theory of desirable gambles}. Instead of requiring the subject to desire all nonnegative gambles, we only require her to desire gambles for which she can efficiently determine the nonnegativity (in particular  sum-of-squares  polynomials). 
We refer to this new criterion as \textit{bounded rationality}.
\end{abstract}


\end{frontmatter}

\section{Introduction}
The subjective foundation of probability by \citet{finetti1937}    is  based on the notion   of rationality (coherence or equiv.\ self-consistency).
A subject is considered rational if she chooses her odds so that there is no bet that leads her to a sure loss (no Dutch books are possible). In this way, since odds are the inverse of probabilities,  de Finetti provided a justification of Kolmogorov's axiomatisation of probability as a rationality criterion on a gambling system.\footnote{De Finetti actually considered only finitely additive probabilities, while $\sigma$-additivity is assumed in Kolmogorov's axiomatisation.}

 Later \citet{williams1975} and \citet{walley1991} showed that it is possible to justify probability in a simpler and more elegant way. 
 This approach  is nowadays known as the \textit{theory of desirable gambles}.\footnote{In this paper, we will refer in particular to the theory of almost desirable gambles.}  To understand this gambling framework, we introduce a subject, Alice, and an experiment whose result  $\omega$ belongs to a   possibility space $\pspace$ (e.g., the experiment may be tossing a coin or determining the future value of a derivative instrument). When Alice is uncertain about the  result $\omega$ of the experiment, we can  model her beliefs about 
this value by asking her whether she accepts  engaging in certain risky transactions, called \textit{gambles}, whose outcome depends on the actual
outcome of the experiment $\omega$.  Mathematically, a gamble is a bounded real-valued function on $\pspace$, $g:\pspace 
\rightarrow \reals$, and if Alice accepts a gamble $g$, this means that she commits herself to 
receive  $g(\omega)$ utiles\footnote{A  theoretical unit of measure of utility, for indicating a supposed quantity of satisfaction derived from an economic 
transaction. It is expressed in some linear utility scale} if the experiment is performed and if the outcome of the experiment eventually happens 
to be the event $\omega \in \pspace$. Since $g(\omega)$ can be negative, Alice can also lose utiles  and hence the desirability of a gamble depends on Alice's beliefs about $\pspace$. Denote by $\gambles$ the set of all the gambles on
$\pspace$. Alice examines gambles in $\gambles$ and comes up with the subset $\domain$ of the gambles that she finds desirable. How can we
characterise the rationality of the assessments represented by $\domain$?

Two obvious  rationality criteria are:  Alice should always accept (do not accept)  gambles  that are nonnegative (negative), because no matter the result of the experiment she never (always) decreases  her utiles. But there is a world of difference between saying and doing.
For instance, let us  consider an infinite space of possibilities like $\Omega=\Reals^2$ and the gamble:
$g(x_1,x_2)=4x_1^4+4x_1^3x_2-3 x_1^2x_2^2+5x_2^4$. 
Should Alice accept this gamble? In practice the answer to this question does not only depend on Alice's beliefs about 
the value of  $x_1$ and $x_2$. We can in fact verify that the above polynomial can be rewritten as $(2  x_1^2 - 2 x_2^2 + x_1 x_2)^2 + ( x_2^2 + 2 x_1 x_2)^2$ and, thus, is always nonnegative. Hence, rationality implies that Alice should always accept it.
However, in these cases, we must also take into account the inherent difficulty of the problem faced by Alice when she wants to determine whether a given gamble is nonnegative or not.  In other words,  we need to quantify the computational complexity needed to address rationality.\\
The aim of this paper is to develop a \textit{computable theory of desirable gambles} by relaxing the  rationality criteria discussed above.
In particular, instead of requiring Alice to accept all nonnegative gambles, we only require Alice to accept gambles
for which she can efficiently determine the nonnegativity. 
We call this  new criterion \textit{bounded rationality}. The term bounded rationality was proposed by Herbert A. \citet{simon1957models} -- it
is the idea that when individuals make decisions, their rationality is limited by the tractability of the decision problem, the cognitive limitations of their minds, and the time available to make the decision. Decision-makers in this view act as ``satisficers'', seeking a satisfactory solution rather than an optimal one.
We do not propose our model as a realistic psychological model of Alice's behaviour, but we embrace the idea that the actual rationality of an agent is determined by its computational intelligence.

In this paper, we  exploit the  results on SOS polynomials  and theory-of-moments relaxation   to make numerical inferences in our theory of bounded rationality and to show that the theory of bounded rationality can be used to approximate the theory of desirable gambles. 
  At the same time,  we provide a gambling interpretation of SOS optimization. Some  applications of the theoretical ideas presented in this paper  can be found  in  \citet{lasserre2009moments,benavoli2016a,benavoli2016c}.
For instance, \citet{benavoli2016a} use this approach to derive a novel set-membership filtering algorithm  for nonlinear polynomial dynamical systems.
Although their approach is not directly formulated as a theory of bounded rationality,  SOS polynomials are used to propagate a set of probability measures in a computational efficient way through the dynamics of a nonlinear system. 
It is worth mentioning that a relaxation of the rationality criteria for desirability has also been investigated by  \citet{schervish2000sets,pelessoni20162}.
The first work focuses on  relaxations of the ``avoiding sure loss'' axiom, while  the second one  focuses    on  
two different criteria (additivity and positive scaling).

A preliminary version of this work appeared in \citep{Benavoli2017b}, but it includes an incorrect statement of duality. This has led us to re-evaluate the whole theory, resulting in a new definition of bounded rationality that we will present in the current manuscript.


\section{Theory of desirable gambles}
\label{sec:TDG}
%
In this section, we briefly introduce the theory of desirable gambles. Let us denote by $\gambles^+ =\{g\in\gambles: g\geq0\} $ the 
subset of the \emph{nonnegative gambles} and with $\domain \subset \gambles$  the subset of the gambles that Alice finds desirable. How can we
characterise the rationality of the assessments in $\domain$?

\begin{definition}
We say that  $\domain$ is a coherent set of (almost) desirable gambles (ADG) 
when it satisfies the following rationality criteria:
\begin{description}
 \item[A.1] If $g \in \gambles^+ $ then $ g\in \domain$ (Accepting Partial Gains);
 \item[A.2] If $ g\in \domain$ then $\sup g \geq0$ (Avoiding Sure Loss);
 \item[A.3] If $ g\in \domain$ then $\lambda g \in \domain$ for every $\lambda>0$ (Positive Scaling);
 \item[A.4] If $ g,h\in \domain$ then $g+h \in \domain$ (Additivity);
 \item[A.5] If $ g+\delta\in \domain$ for every $\delta>0$ then $g \in \domain$ (Closure).
\end{description}
\end{definition}
 The criterion A.5  does not actually follow from rationality and can be  omitted \citep{seidenfeld1990,walley1991,miranda2010c}. 
 However, it is useful to derive a connection between  the theory of desirable gambles and probability theory and for this reason we consider it in this paper. This connection will be briefly discussed in Section \ref{sec:dualityADG}.

To explain  these rationality criteria, let us introduce a simple example:  the toss of a fair coin  $\Omega=\{Head,Tail\}$.
 A gamble $g$ in this case has two components $g(Head)=g_1$ and $g(Tail)=g_2$.
 If Alice accepts $g$ then  she commits herself to receive/pay $g_1$ if  the outcome is Heads and   $g_2$ if Tails.
Since a gamble is in this case an element of $\mathbb{R}^2$,  $g=(g_1,g_2)$, we can plot the gambles  Alice accepts  in a 2D coordinate system with coordinate  $g_1$ and $g_2$. 
 
A.1 says that Alice is 
willing to accept any gamble $g=(g_1,g_2)$ that, no matter the result of the experiment, may increase her wealth without ever decreasing it, that is with $g_i \geq 0$ -- Alice always accepts the first quadrant, Figure~\ref{fig:coin0}(a). 
Similarly.  Alice does not accept  any gamble $g=(g_1,g_2)$ that will surely decrease her wealth, that is with  $g_i<0$. In other words, Alice always does not accept  the interior of the third quadrant, Figure~\ref{fig:coin0}(b).  This is the meaning of A.2. Then we ask Alice about $g=(-0.1,1)$ -- she loses $0.1$ if Heads and wins $1$  if Tails.
Since Alice knows that the coin is fair, she accepts this gamble as well as all the gambles of the form  $\lambda g$ with  $\lambda>0$, because this is just a ``change of currency'' (this is A.3).  Similarly, she accepts all the gambles $ g + h$ for any  $h \in \gambles^+$, since these gambles are even more favourable for her (this  is basically A.4). Now, we can ask Alice about $g=(1,-0.1)$ and the argument is symmetric to the above case.
We therefore obtain the following set of desirable gambles (see Figure~\ref{fig:coin0}(c)):
$\domain_2=\{g \in \mathbb{R}^2 \mid 10g_1+g_2\geq 0 \text{ and } g_1+10g_2\geq 0\}$.
Finally, we can ask Alice about $g=(-1,1)$ -- she loses $1$ if Heads and wins $1$  if Tails.
Since the coin is fair, Alice may accept or not accept this gamble. A.5 implies that she  must accept it (closure). A similar conclusion can be derived for the symmetric gamble $g=(1,-1)$. Figure~\ref{fig:coin0}(d) is her final set of desirable gambles about the experiment concerned with the toss of a fair coin, which in a formula becomes
$\domain_3=\{g \in \mathbb{R}^2 \mid g_1+g_2\geq 0\}$.
Alice does not accept any other gamble. In fact, if Alice  would also accept for instance $h=(-2,0.5)$ then, 
since she has also accepted $g=(1.5,-1)$, i.e., $g\in \domain_3$, she must also accept $g+h$ (because this gamble is also favourable to her).
However, $g+h=(-0.5,-0.5)$ is always negative, Alice always loses utiles in this case. In other words, by accepting $h=(-2,0.5)$ Alice incurs a sure loss -- she is irrational ({A.2} is violated).

In this example, we can see that Alice's set of  desirable gambles is a closed half-space, but this does not have  to be always  the case.
For instance, if Alice does not know anything about the coin, she should only accept nonnegative gambles: $\domain=\gambles^+$.
This corresponds to a state of complete ignorance, but all intermediate cases from complete beliefs on the probability of the coin to complete ignorance are possible. In general, $\domain$ is a pointed (whose vertex is the origin) closed convex cone that includes $\gambles^+$ and excludes the interior of the negative orthant (this follows by   A.1--A.5).

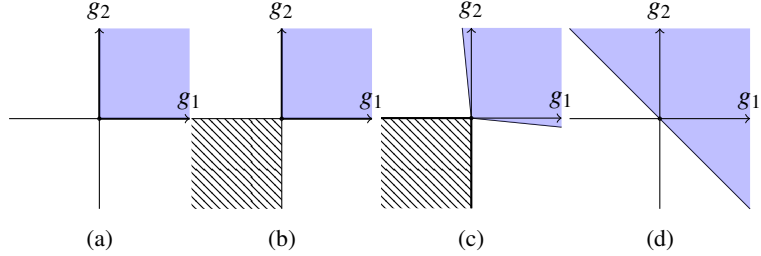
\begin{figure*}[t!]
    \centering
    \begin{subfigure}[t]{0.2\textwidth}
        \centering
\begin{tikzpicture}[scale=1.2]
    \draw[->] (-1,0) coordinate (xl) -- (1,0) coordinate (xu) node[above] {$g_1$};
    \draw[->] (0,-1) coordinate (yl) -- (0,1) coordinate (yu) node[above] {$g_2$};
        \draw (0,0) circle (0.5pt);
    \begin{pgfonlayer}{background}
      \draw[border] (0,0) -- (0,1) coordinate (a1away);
      \draw[border] (0,0) -- (1,0) coordinate (a2away);
      \fill[blue!50,  opacity=0.5] (0,0) -- (1,0) -| (1,1) --(0,1);
    \end{pgfonlayer}
  \end{tikzpicture}
        \caption{}
    \end{subfigure}%
    \begin{subfigure}[t]{0.2\textwidth}
        \centering
\begin{tikzpicture}[scale=1.2]
    \draw[->] (-1,0) coordinate (xl) -- (1,0) coordinate (xu) node[above] {$g_1$};
    \draw[->] (0,-1) coordinate (yl) -- (0,1) coordinate (yu) node[above] {$g_2$};
    \draw (0,0) circle (0.5pt);
        \begin{pgfonlayer}{background}
      \draw[border] (0,0) -- (0,1) coordinate (a1away);
      \draw[border] (0,0) -- (1,0) coordinate (a2away);
      \fill[blue!50,  opacity=0.5] (0,0) -- (1,0) -| (1,1) --(0,1);
    \end{pgfonlayer}
    \begin{pgfonlayer}{background}
      \fill[pattern=north west lines, pattern color=black] (0,0) -- (-1,0) -| (-1,-1) --(0,-1);
    \end{pgfonlayer}
  \end{tikzpicture}
        \caption{}
    \end{subfigure}
        \begin{subfigure}[t]{0.2\textwidth}
        \centering
 \begin{tikzpicture}[scale=1.2]
    \draw[->] (-1,0) coordinate (xl) -- (1,0) coordinate (xu) node[above] {$g_1$};
    \draw[->] (0,-1) coordinate (yl) -- (0,1) coordinate (yu) node[above] {$g_2$};
    \draw (0,0) circle (0.5pt);
    \begin{pgfonlayer}{background}
      \draw[border] (0,0) -- (0,-1) coordinate (a1away);
      \draw[border] (0,0) -- (-1,0) coordinate (a2away);
      \fill[pattern=north west lines, pattern color=black] (0,0) -- (-1,0) -| (-1,-1) --(0,-1);
    \end{pgfonlayer}
    \begin{pgfonlayer}{background}
      \draw[-] (0,0) -- (-0.1,1) coordinate (a1away);
      \draw[-] (0,0) -- (1,-0.1) coordinate (a2away);
      \fill[blue!50,  opacity=0.5] (0,0) -- (a1away) -| (a2away) --(0,0);
    \end{pgfonlayer}
  \end{tikzpicture}
          \caption{}
    \end{subfigure}
        \begin{subfigure}[t]{0.2\textwidth}
        \centering
  \begin{tikzpicture}[scale=1.2]
    \draw[->] (-1,0) coordinate (xl) -- (1,0) coordinate (xu) node[above] {$g_1$};
    \draw[->] (0,-1) coordinate (yl) -- (0,1) coordinate (yu) node[above] {$g_2$};
    \draw (0,0) circle (0.5pt);
    \begin{pgfonlayer}{background}
      \draw[-] (0,0) -- (-1,1) coordinate (a1away);
      \draw[-] (0,0) -- (1,-1) coordinate (a2away);
      \fill[blue!50,  opacity=0.5] (0,0) -- (a1away) -| (a2away) --(0,0);
    \end{pgfonlayer}
  \end{tikzpicture}
        \caption{}
    \end{subfigure}
    \caption{Alice's sets of coherent almost desirable gambles for the experiment of tossing a fair coin.  \label{fig:coin0}}
\end{figure*}
%
For the coin, the space of possibilities is finite and in this case Alice can check if a gamble $g$ is nonnegative by simply examining the elements of the vector $g$. In this paper, we are interested in  infinite spaces, in particular   $\Omega=\reals^n$, where applying the above rationality criteria is far from easy.  We aim at developing a theory of \textit{bounded rationality} for this case. Before doing that, we briefly recall the connection  between ADG and probability theory.

\section{Duality for ADG}
\label{sec:dualityADG}

Duality can be defined for general spaces of possibilities $\pspace$ \citep{walley1991}.  However, for the purpose of the present paper, we 
consider gambles that are bounded  real-valued function on $\reals^n$, i.e., $g:\reals^n\rightarrow \reals$.
Let $\mathcal{A}$ be an algebra of subsets of $\reals^n$  and $\mu:\mathcal{A} \rightarrow [-\infty,\infty]$ 
denotes a charge: that is $\mu$ is a  finitely additive set function of $\mathcal{A}$  \citep[Ch.11]{aliprantisborder}.
Let $\mathcal{A}_{\reals}$ denote the algebra generated in $\reals$ by the collection of all half open intervals
 \citep[Th.11.8]{aliprantisborder}:
\begin{theorem}
\label{th:boundedint}
 Every bounded $(\mathcal{A},\mathcal{A}_{\reals})$-measurable function is integrable w.r.t.\ any finite charge.
\end{theorem}
Therefore, for any bounded $(\mathcal{A},\mathcal{A}_{\reals})$-measurable function $g$ and finite charge $\mu$ we can define  $\int gd\mu$, which we can interpret as
a linear functional  $L(\cdot):=\langle \cdot, \mu\rangle$ on  bounded $(\mathcal{A},\mathcal{A}_{\reals})$-measurable gambles $g$ (with some abuse of notation, we denote the set of  bounded $(\mathcal{A},\mathcal{A}_{\reals})$-measurable functions with $\gambles$).
We denote by $\mathcal{M}$ the set of all finite charges on $\mathcal{L} $ and  by $\mathcal{M}^+$ the set of nonnegative charges.
A linear functional of gambles is said to be \emph{nonnegative} whenever it  satisfies : $L(g) \geq 0$, for $g \in \mathcal{L}^+$. A nonnegative linear functional is called a \emph{state} if moreover it preserves the unitary constant gamble. Hence, in this context, since $\langle 1, \mu\rangle=\int 1 d\mu=1$,  the set of states $\mathscr{S}$ corresponds to the set of  all probability charges.
We define the \emph{dual} of a subset $\mathcal{D}$ of $\mathcal{L}$ as:\\
\begin{equation}
\mathcal{D}^\bullet=\left\{\mu\in \mathcal{M}: \int gd\mu \geq0, ~\forall g \in \mathcal{D}\right\}.
\end{equation}
Similarly, the dual of a subset $\mathcal{R}$ of $\mathcal{M}$ is the set:\\
\begin{equation}
\mathcal{R}^\bullet=\left\{g\in \mathcal{L}: \int gd\mu \geq0, ~\forall \mu \in \mathcal{R}\right\}.
\end{equation}
Note that in both cases $(\cdot)^\bullet$ is always a closed convex cone \citep[Lem.5.102(4)]{aliprantisborder}. Furthermore,  one has that $(\cdot)^\bullet{^\bullet}=(\cdot)$, whenever $(\cdot)$ is a closed convex cone  \citep[Th.5.103]{aliprantisborder}, and   $(\cdot)_1 \subseteq (\cdot)_2$ if and only if $(\cdot)_2^\bullet \subseteq (\cdot)_1^\bullet$ \citep[Lem.5.102(1)]{aliprantisborder}. In particular, whenever $(\cdot)_1$ and $(\cdot)_2$ are closed convex cones, $(\cdot)_1 \subsetneq (\cdot)_2$ if and only if $(\cdot)_2^\bullet \subsetneq (\cdot)_1^\bullet$.

Based on those facts, it is thus possible to verify that the dual of a coherent set of desirable gambles can actually be completely described in terms of a (closed convex) set of states (probability charges). 
In this aim, we start by the following observations.

\begin{proposition}\label{prop:czesc}
 It holds that 
 \begin{enumerate}
   \item $(\mathcal{L})^\bullet=\{0\}$ and $\mathcal{L}=(\{0\})^\bullet$;
 \item $(\mathcal{L}^+)^\bullet=\mathcal{M}^+$ and $\mathcal{L}^+=(\mathcal{M}^+)^\bullet$;
  \end{enumerate}
\end{proposition}
\begin{proof} Since $(\cdot)^\bullet{^\bullet}=(\cdot)$, whenever $(\cdot)$ is a closed convex cone, in both cases it is enough to verify only one of the claims.
For the first item, the second claim is immediate.
For the second item, we verify the first claim.
The inclusion from right to left being clear, for the other direction  observe that: (i) $ g=I_{\{B\}}$ (with $I_B$ being the indicator function on $B \in \mathcal{A}$), is a nonnegative gamble and, therefore, is in $\mathcal{L}^+$; (ii) if $\mu$ is negative in $B \in \mathcal{A}$, i.e.,  then $\int I_{B}d\mu$ is negative too and, thus, $\mu$ cannot be in $(\mathcal{L}^+)^\bullet$.  
\end{proof}
\begin{proposition}\label{prop:czesc2}
Let $\domain $ be a closed convex cone. The following claims are equivalent
 \begin{enumerate}
   \item $\domain$ is coherent;
     \item $\domain\supseteq \mathcal{L}^+$ and $\domain \neq \mathcal{L}$;
   \item $(\domain)^\bullet\subseteq \mathcal{M}^+$ and $(\domain)^\bullet \neq \{0\}$.
  \end{enumerate}
\end{proposition}
\begin{proof} (2) $\Leftrightarrow$ (3): From Proposition \ref{prop:czesc}, $(\domain)^\bullet = \{0\}$ if and only if $\domain = \mathcal{L}$, and $\domain\supseteq \mathcal{L}^+$ if and only if $(\domain)^\bullet\subseteq \mathcal{M}^+$.
\\ (1) $\Rightarrow$ (2): Assume $\domain$ is coherent. By A.1 $\domain \supseteq \mathcal{L}^+$ and by A.2 there is $g \in \mathcal{L}$ such that $\sup g < 0$ and $g \notin \domain$. 
\\ (2) $\Rightarrow$ (1): Let $\mathcal{L}^+ \subseteq \domain \subsetneq \mathcal{L}$. First of all, notice that, by Proposition \ref{prop:czesc}, $(\domain)^\bullet\subseteq (\mathcal{L}^+)^\bullet = \mathcal{M}^+$. 
Now, assume that $\domain$ is not coherent. This means A.2 fails, i.e. there is $g \in \mathcal{L}$ such that $\sup g < 0$ and $g \in \domain$. 
Hence, consider $\mu \in \mathcal{M}^+$ and pick $g \in \domain$ such that $\sup g < 0$. It holds that $\langle g, \mu\rangle \geq 0$ if and only if $\mu=0$, meaning that $(\domain)^\bullet=\{0\}$ and therefore, by Proposition \ref{prop:czesc} again, $\domain=\mathcal{L}$, a contradiction.
\end{proof}

Hence, the following theorem holds.
\begin{theorem}
\label{prop:dualcharges0}
The map
\[\domain \mapsto  \mathcal{P}:=\domain^\bullet \cap \mathscr{S} \]
establishes a bijection between coherent sets of desirable gambles and closed convex sets of states. 
\end{theorem}
\begin{proof} The proof is analogous to that by \citet[Th.4]{pmlr-v62-benavoli17b}.
 Let $\domain$ be a coherent set of desirable gambles. By Proposition \ref{prop:czesc2}, we get that $\domain^\bullet$ is a closed convex cone included in $\mathcal{M}^+$ that does not reduce to the origin.  
Thus, after normalisation, $\mathcal{P}$ is nonempty. Preservation of closedness and convexity by finite intersections yields that $\mathcal{P}$ is closed and convex. Furthermore $\mathbb{R}_+\mathcal{P}=\domain^\bullet$, and therefore $\domain= (\mathbb{R}_+\mathcal{P})^\bullet$, where $\mathbb{R}_+\mathcal{P}:=\{ \lambda\mu : \lambda \geq 0, \mu \in \mathcal{P}\}$, meaning that the map is an injection.
We finally verify that the map is also a surjection. To do this, let $\mathcal{P}$ be a non empty closed convex set of probability charges. It holds that $\mathbb{R}_+\mathcal{P}$ is a closed convex cone included in $\mathcal{M}^+$ different from $\{0\}$. Again by Proposition \ref{prop:czesc2},  we  conclude that the dual $(\mathbb{R}_+\mathcal{P})^\bullet$ of $\mathbb{R}_+\mathcal{P}$ is a coherent set of desirable gambles and $\mathcal{P}=\mathbb{R}_+\mathcal{P}\cap \mathscr{S}= (\mathbb{R}_+\mathcal{P}){{}^\bullet{^\bullet}}\cap \mathscr{S}$. 
\end{proof}

We therefore identify the dual of $\domain$ with
\begin{equation}
 \mathcal{P}=\left\{\mu\in \mathcal{M}^+: \int gd\mu \geq0, ~\int d\mu =1,~\forall g \in \domain\right\},
\end{equation} 
which is a closed convex-set of probability charges.

\subsection{Unbounded gambles}
\label{sec:unbounded gambles}
In this paper, we will also consider unbounded real-valued functions in $\reals^n$
and, therefore, we need to introduce another definition of measurability.
Let  $\mathcal{M}$ denote the space of finite signed Borel measures on $\reals^n$, whose positive cone
$\mathcal{M}^+$ is the space of  finite Borel measures $\mu$ on  $\reals^n$.
Let $\gambles$ be the set of  real-valued functions on  $\reals^n$ that are integrable with respect to every measure 
$\mu \in \mathcal{M}^+$ and $\gambles^+$ be the cone of non-negative integrable functions.
With an abuse of terminology, we call these functions gambles.\footnote{For a more general extension of the theory of desirable gambles to unbounded gambles see  \citet{troffaes2003extension}.}
Also in this case, we can  define a cone of desirable gambles in $\gambles$ satisfying the properties A.1--A.5.
By defining the set of states as $\mathscr{S}:=\{\mu \in \mathcal{M}: \int d\mu =1\}$, we can  prove the following.
\begin{theorem}
\label{prop:dualcharges}
The map
\[\domain \mapsto  \mathcal{P}:=\domain^\bullet \cap \mathscr{S} \]
establishes a bijection between coherent sets of desirable gambles and closed convex sets of states. 
\end{theorem}
The proof is identical to that for Theorem \ref{prop:dualcharges0}. Here, $\mathcal{P}$ is the set of all probability measures on  $\reals^n$.
In the sequel, we will refer to this duality when we will consider unbounded gambles.

%


\section{Finite assessments}
The goal of this and next sections is to define a practical notion of desirability.
To this end, we first assume that the set of gambles that Alice finds desirable is finitely  generated. By this, we mean that there is 
a finite set of gambles $G=\{ g_1, \dots, g_{|G|}\}$ such that $\domain=\posi(G \cup \gambles^+)$, where  the $\posi$ of a set $A \subset \gambles$ is
defined as
\begin{equation}
\label{eq:posi}
\posi(A):=\set{\sum_{j=1}^{|A|}\lambda_jg_j}{g_j\in A,\lambda_j\geq 0},
\end{equation} 
and where by $|A|$ we denote the cardinality of the set $A$.
By using this definition, it is clear that whenever $\domain$ is finitely generated, it includes all nonnegative gambles
and satisfies A.3, A.4 and A.5. Once Alice has defined $G$ and so $\domain$ via $\posi$,  ADG assumes that she is able to perform the following operations: to check that $ \domain$ avoids sure loss (A.2 is also satisfied);  to determine the implication of desirability.
 It is easy to show that all above operations in ADG imply the assessment of the nonnegativity of a gamble.

\begin{proposition} 
\label{th:noneg} Given a finite set $G \subset \gambles$ of desirable gambles, the set $\posi(G \cup \gambles^+)$  includes the gamble $f$ if and only if  there exist $\lambda_j\geq 0$ for $j=1,\dots,|G|$ such that
\begin{equation}
\label{eq:feasibil0}
f- \sum\limits_{j=1}^{|G|}\lambda_jg_j\geq0.
\end{equation} 
\end{proposition}
  There are two  subcases of \eqref{eq:feasibil0} that are particularly interesting.   
 The first  is when  $f=h-\lambda_0$ for some $\lambda_0 \in \mathbb{R}$. That allows us to define the concept of lower prevision
 \citep{walley1991,miranda2008a}.
  \begin{definition}
  Assume that $\domain=\posi(G \cup \gambles^+)$ is an ADG, then the solution of the following problem
  \begin{equation}
\label{eq:lowerpre}
{\begin{array}{l}
\sup\limits_{\lambda_0 \in \mathbb{R},\lambda_j\geq0} \lambda_0, ~~~~s.t.~~~h-\lambda_0- \sum\limits_{j=1}^{|G|}\lambda_jg_j\geq0,
\end{array}}
\end{equation} 
is called the lower prevision of $h$ and denoted as $\underline{E}[h]$.
 \end{definition}
 From a behavioural point of view, we can reinterpret this by saying that Alice is willing to buy gamble $h$ at price $\lambda_0$, since she is giving
away $\lambda_0$ utiles while gaining $h$. The lower prevision is the supremum buying price  for $h$. We can equivalently define the upper prevision of $h$ as $\overline{E}[h]=-\underline{E}[-h]$. From Section \ref{sec:dualityADG}, it can be easily shown that $\underline{E}[h]$ is the lower expectation of $h$
computed w.r.t.\ the probability charges (or measures if we consider the case in Section \ref{sec:unbounded gambles}) in $ \mathcal{P}$. {As a matter of fact, the dual of \eqref{eq:lowerpre} is the moment problem: 
  \begin{equation}
\label{eq:lowerpredual}
\begin{array}{l}
\inf\limits_{\mu \in \mathcal{M}^+} \int{hd\mu} ~~~~~~s.t.~~~\int{ d\mu}=1, ~~ \int{g_j d\mu}\geq0 \text{ for $j=1,\dots,|G|$}.
\end{array}
\end{equation} 
 
 \begin{example}[Markov's inequality]
 \label{eq:markov1}
  Consider the nonnegative function $x$ on $\Omega=[0,x_{max}]$ with $x_{max} \in \reals^+$ and assume that Alice 
  finds the gambles $g=x-m$ and $-g=-x+m$ to be desirable for some $m\in \Omega$, i.e., 
  $$
  G=\{g,-g\}.
  $$
  Consider the event $x\geq u$ for some $u \in \mathbb{R}^+$. Our goal is to determine
Alice's upper prevision (infimum selling price) for this event, equivalently for the gamble $I_{\{[u,\infty)\}}$.
We  need to apply  \eqref{eq:lowerpre} which in this case can be written as:
\begin{equation}
 \label{eq:lowerpreMI}
\begin{array}{l}
\overline{E}(I_{\{[u,\infty)\}})=\inf\limits_{\lambda_{1j}\leq0,\lambda_0\in \mathbb{R}} ~~\lambda_0\\
  s.t.\\
 \lambda_0+ \lambda_{11}(x-m) - \lambda_{12}(x-m) \geq I_{\{[u,\infty)\}}(x), ~~\forall x \in \Omega.\\
 \end{array}
\end{equation}
By defining $\lambda_1=\lambda_{11}-\lambda_{12}$, which now spans $\mathbb{R}$, we can rewrite  \eqref{eq:lowerpreMI} as:
\begin{equation}
 \label{eq:natelowmark1}
\begin{array}{l}
  \overline{E}(I_{\{[u,\infty)\}})=\inf\limits_{\lambda_j\in \mathbb{R}} ~~\lambda_0\\
  s.t.\\
 \lambda_0+ \lambda_1 (x-m)  \geq 1, ~~\forall x \in \Omega:x\geq u,\\
 \lambda_0+ \lambda_1 (x-m)  \geq 0, ~~\forall x \in \Omega:x< u.\\
 \end{array}
\end{equation}
It can be  seen that $\overline{E}(I_{\{[u,\infty)\}})=1$ whenever $u\leq0$ and $\overline{E}(I_{\{[u,\infty)\}})=0$ whenever $u\geq x_{max}$.
In the other cases, the above problem must satisfy:
\begin{equation}
 \label{eq:natelowmark2}
\begin{array}{l}
  \overline{E}(I_{\{[u,\infty)\}}) \leq \inf\limits_{\lambda_i\in \mathbb{R}} ~~\lambda_0\\
  s.t.\\
 \lambda_0+ \lambda_1 (u-m)  \geq 1,\\
 \lambda_0+ \lambda_1 (-m)  \geq 0,\\
 \end{array}
\end{equation}
where we have considered the worst cases for the inequalities. If $m<u$, the infimum is obtained when the inequalities in (\ref{eq:natelowmark2}) are equalities and  is equal to $\frac{m}{u}$. When $m\geq u$, the infimum is obtained for $\lambda_0=1$, $\lambda_1=0$.
We have therefore derived  Markov's inequality:
$$
P(x\geq u) \leq  \min\left(1,\dfrac{m}{u}\right)=\min\left(1,\dfrac{E[x]}{u}\right).
$$
The last equality follows from the fact that \eqref{eq:lowerpreMI} is equivalent to (see \eqref{eq:lowerpredual}):
  \begin{equation}
\label{eq:lowerpredualMI}
\begin{array}{l}
\sup\limits_{\mu \in \mathcal{M}^+} \int{I_{\{[u,\infty)\}} d\mu} ~~~~~~s.t.~~~\int{ d\mu}=1, ~~ \int{x d\mu}= m,
\end{array}
\end{equation} 
and so $m$  is just the expectation of $x$, i.e.,  $m=E[x]=\int{x d\mu}$.

\end{example}

 The second  subcase  allows us to formulate sure loss as  nonnegativity of a gamble  \citep[Alg. 2]{walley2004direct}. 
 
 \begin{proposition}
 Let us consider   $\domain=\posi(G \cup \gambles^+)$ and the following problem:
  \begin{equation}
\label{eq:feasibil0proof0}
\begin{array}{l}
\sup\limits_{0\leq \lambda_0\leq1,~\lambda_j\geq0} \lambda_0, ~~~~s.t.~~~ -\lambda_0- \sum\limits_{j=1}^{|G|}\lambda_jg_j\geq0.
\end{array}
\end{equation} 
 $\domain$ incurs a sure loss iff the above problem has solution $\lambda_0^*=1$ and avoids sure loss iff $\lambda^*_0=0$.
 \end{proposition}
\begin{proof}
We briefly sketch the proof (see \citealt{walley2004direct}).
Assume that $\domain$ incurs a sure loss, then there exist $\lambda_j\geq 0$ for $j=1,\dots,|G|$ such that
 $\sum_{j=1}^{|G|}\lambda_jg_j<0$ and, thus,  $-\sum_{j=1}^{|G|}\lambda_jg_j>0$.
 Then we can increase $\lambda_0$ as much as we want.
  Similarly, we can prove the inverse implication. 
\end{proof}

 \begin{example}[Markov's inequality cont.]
  Consider again the previous example, i.e., 
  $$
  G=\{x-m,-x+m\}.
  $$
We want to show when/if $\domain$ incurs a sure loss. Consider  \eqref{eq:feasibil0proof0} and assume   $m\notin \Omega$. 
    \begin{equation}
\label{eq:feasibil0proof0MI}
\begin{array}{l}
\sup\limits_{0\leq \lambda_0\leq1,~\lambda_1 \in \reals} \lambda_0, ~~~~s.t.~~~ -\lambda_0- \lambda_1(x-m)\geq0 ~~x\in \Omega.
\end{array}
\end{equation} 
Note that if $m< 0$ then $x-m>0$ in $\Omega$ and so $- \lambda_1(x-m)\geq0 $ provided that $\lambda_1\leq0$. Hence, the solution of the above optimisation
problem is $\lambda_0=1$ (sure loss).
Similarly, if $m>x_{max}$ then  $- \lambda_1(x-m)\geq0 $  provided that $\lambda_1\geq0$ and so the optimum is  $\lambda_0=1$.
Assume now that $m \in \Omega$, then the best solution is obtained for $\lambda_1=0$ and so the optimum is $\lambda_0=0$ and, therefore,
$\domain$ avoids sure loss.
We can conclude that, to avoid a sure loss, Alice should  accept the gambles $x-m,-x+m$ only when
$$
\inf_{x\in \Omega} x=0 ~~\leq ~~ m ~~\leq~~\sup_{x\in \Omega} x=x_{max}.
$$
In the previous example we saw that  $m=E[x]=\int{x d\mu}$, hence the above inequalities mean  that 
$$
\inf_{x\in \Omega} x ~~\leq ~~ E[x]~~\leq~~\sup_{x\in \Omega} x.
$$
\end{example}

 \subsection{Complexity of inferences}
 When $\Omega$ is finite (e.g., coin toss), then a gamble $g$ can also be seen as a vector in $\reals^{|{\Omega}|}$ (where $|{\Omega}|=2$ for the coin).  Then \eqref{eq:feasibil0} can be expressed as a linear programming problem, thus its  complexity is polynomial: \textit{Alice can determine the implication of her assessments of desirability in polynomial time}. In case $\Omega=\reals^n$, $f:\reals^n \rightarrow \reals$, solving  \eqref{eq:feasibil0}  means to check the existence
of {real parameters $\lambda_j\geq0$ ($j=1,\dots,|G|$)} such that the function   
\begin{equation}
\label{eq:feasibil}
F:=f- \sum_{j=1}^{|G|}\lambda_jg_j
\end{equation}
is nonnegative in $\reals^n$.
In order to study the problem from a computational viewpoint, and avoid undecidability results, it is
clear that  we must impose further restrictions on the class of functions $F$. 
At the same time we would like to keep the problem general enough, in order not to lose expressiveness of the model. 
A good compromise can be achieved by considering the case of multivariate polynomials. The decidability of $F\geq0$ 
for multivariate polynomials can be proven by means of the Tarski-Seidenberg quantifier elimination theory \citep{tarski1951decision,seidenberg1954new}.

Let $d \in \N$. By $\reals_{2d}[x_1]$ we denote the set of all polynomials up to degree $2d$  in the indeterminate variable $x_1 \in \reals$ with real-valued coefficients.
With the usual definitions of addition and scalar multiplication, $\reals_{2d}[x_1]$ becomes a vector space over the field $\reals$ of real numbers. 
We can introduce a basis for $\reals_{2d}[x_1]$ that we denote as $v_{2d}(x_1)$ where
$v_{j}(x_1)=[1,x_1,x_1^2,\dots,x_1^{j}]^\top$.
We denote the dimension of $v_{j}(x_1)$ as $s_1(j)$ for $j=0,1,2,\dots$, e.g., $s_1(2d)=2d{+1}$. Any polynomial in $\reals_{2d}[x_1]$ can then be written as $p(x_1)=b^\top\,v_{2d}(x_1)$ being $b\in \reals^{s_1(2d)}$ the vector of coefficients. One can actually provide a square matrix representation of a polynomial, as for any polynomial in $\reals_{2d}[x_1]$ there is a (non-unique) matrix $Q \in \reals_s^{s_1(d)\times s_1(d)}$ such that $p(x_1)=v^\top_{d}(x_1)Q v_{d}(x_1)$, where $\reals_s^{s_1(d)\times s_1(d)}$ is the space of $s_1(d)\times s_1(d)$ real-symmetric matrices.
In virtue of these observations
we may therefore also be interested in some subsets of  $\reals_{2d}[x_1]$ that are:\\
(1) the subset of nonnegative polynomials, denoted as $\reals^+_{2d}[x_1]$; \\(2) the subset of  polynomials  
\begin{equation} \label{eqn:SOSrepr}
 \Sigma_{2d}[x_1]=\left\{p(x_1) \in \reals_{2d}[x_1] ~\Big|~ p(x_1)=v^\top_{d}(x_1)Q v_{d}(x_1) \text{ with } Q \in \reals_s^{s_1(d)\times s_1(d)}, ~Q\geq0\right\}.
\end{equation}
The polynomials in $\Sigma_{2d}[x_1]$ are also called SOS polynomials. This is because any polynomial in $\reals_{2d}[x_1]$ that is a sum of squares of polynomials belongs to $\Sigma_{2d}[x_1]$ and vice versa \citep[Prop. 2.1]{lasserre2009moments}. 
Clearly any polynomial in $\Sigma_{2d}[x_1]$ is necessarily nonnegative. It is therefore natural to ask whether the two sets, $\reals^+_{2d}[x_1]$ and $\Sigma_{2d}[x_1]$, coincide, and therefore whether Equation (\ref{eqn:SOSrepr}) provides a representation theorem for nonnegative univariate polynomial. The answer to this question is affirmative \citep[Prop. 2.3]{lasserre2009moments}.

The previous framework can be extended to multivariate polynomials $\reals_{2d}[x_1,\dots,x_n]$. Indeed, polynomials
in $\reals_{2d}[x_1,\dots,x_n]$ can be written as $p(x_1,\dots,x_n)=b^\top\,\,v_{2d}(x_1,\dots,x_n)$ with
\begin{align} 
v_{j}(x_1,\dots,x_n)&=[1,x_1,\dots,x_n,x_1^2,x_1x_2,\dots,x_{n-1}x_{n},x_n^2,\dots,x_1^{2d},\dots,x_n^{j}]^\top, 
\end{align}
 $b \in \reals^{s_n(2d)}$ with $s_n(j) = {n+j \choose j}$ for $j=0,1,2,\dots$. Moreover, one can always find a real-symmetric matrix $Q$ such that $p(x_1,\dots,x_n)= v^\top_{d}(x_1,\dots,x_n)Q v_{d}(x_1,\dots,x_n)$.
 Similarly to the univariate case,  we can thence define  the nonnegative  polynomials $\reals^+_{2d}[x_1,\dots,x_n]$  and the SOS polynomials $ \Sigma_{2d}[x_1,\dots,x_n]$.
 In the multivariate case, however positive semi-definite real-symmetric matrices do not necessarily characterise being nonnegative. i.e., it is in general not true that every nonnegative polynomial is SOS. 
 For instance $g(x_1,x_2)=x_1^2x_2^2(x_1^2+x_2^2-1)+1$  is a nonnegative polynomial that does not have a SOS representation \citep[Sec.2.4]{lasserre2009moments}.  \citet{hilbert1888darstellung} showed the following.
 
 \begin{proposition}
\label{prop:nonnRsos}
 $\reals^+_{2d}[x_1,\dots,x_n]=\Sigma_{2d}[x_1,\dots,x_n]$ holds iff
either $n=1$ or $d=1$ or $(n,d)=(2,2)$. 
\end{proposition}

%

The problem of testing global nonnegativity of a polynomial function 
is in general \textit{NP-hard}. If Alice wants to avoid the  complexity  associated with this problem, an alternative option is to consider a subset of polynomials for which a nonnegativity test is not  \textit{NP-hard}. The problem of testing if a given \textit{polynomial} is SOS has polynomial complexity {(we only need to check if the matrix of coefficients $Q$ in \eqref{eqn:SOSrepr} is positive semi-definite), see \cite{lasserre2009moments}.} 
\begin{example}
 Let us consider the  polynomial $ f=\tfrac{1}{4}-x(1-x)$, we want to show that $f$ is SOS.
Let us attempt to rewrite it as 
$$
f=\tfrac{1}{4}-x(1-x)=\sigma_0(x),
$$
for $\sigma_0(x)\in\Sigma_{2d}$ and, therefore, $\sigma_0(x) =[1,x]Q[1,x]^T$ with $Q$ being a $2 \times 2$ positive semi-definite matrix. 
 By equating the coefficients of the polynomials in $\frac{1}{4}-x(1-x)=[1,x]Q[1,x]^T$, we  find the matrix 
 $$
 Q=\begin{bmatrix}
     \tfrac{1}{4} & -\tfrac{1}{2}\\
-\tfrac{1}{2} & 1
     \end{bmatrix}.
 $$
 The matrix is positive semi-definite and, thus, the polynomial $f$ is SOS. 
\end{example}

\section{Bounded rationality}
In the bounded rationality theory we are going to present we will work with $\pspace=\reals^n$ and make two important assumptions. We  assume that $\gambles$ is the set of multivariate polynomials of $n$ variables and  of degree less than or equal to $2d$, with $d\in\N$. We  denote $\gambles$ as $\gambles_{2d}$ and the nonnegative polynomials as $\gambles^+_{2d}$. Note that $\gambles_{2d}$ is a vector space and A.1--A.5 are  well-defined in $\gambles_{2d}$. This restriction is useful to define the computational complexity of our bounded rationality theory as a function of $n$ and $d$.
We now define our bounded rationality criteria, and point out the two assumptions.

\begin{definition}
\label{def:badg}
We say that  $\bdomain \subset \gambles_{2d}$ is a {\bf bounded-rationality}  coherent set of almost desirable gambles (BADG)
when it satisfies A.3--A.5 (i.e. it is a closed convex cone) and:
\begin{description}
 \item[bA.1] If $ g \in \Sigma_{2d} $ then $ g\in \bdomain$ (Bounded Accepting Partial Gain),
 \item[bA.2] If $ g \in \Sigma^-_{2d}$ then $ g\notin \bdomain$ (Bounded Avoiding Sure Loss);
  \end{description}
where  $\Sigma_{2d}\subset \gambles^+_{2d}$ is the set of SOS {of degree less than or equal to $2d$} and  $\Sigma^-_{2d}:=\{  g \in \gambles_{2d} \mid \sup g < 0, -g \in \Sigma_{2d} \}$ is the set of negative SOS polynomials of degree less than or equal to $2d$ (or stated otherwise, it is the interior of $-\Sigma_{2d}$).
\end{definition}
We have seen that A.1 implies that a coherent set of gambles must include all nonnegative gambles (and, therefore, $\gambles^+_{2d}$ that is the set of all nonnegative polynomials) and that, additionally, A.2 means that it should not include negative gambles. Here, we restrict A.1 and A.2 imposing bounded-rationality that implies that the set must only include SOS polynomials and avoid negative SOS polynomials, 
up to degree $2d$. In  BADG theory, we ask Alice only to always accept gambles  for which she can efficiently determine the nonnegativity (SOS polynomials) and to never accept gambles for which she can efficiently determine the negativity (negative SOS polynomials). 
Using the terminology from 
\cite[Definition 1]{de2012exchangeability}\footnote{Notice that the authors use a different notion of coherence: they do not assume the closure condition (A.5), and they would require that both the zero gamble and gambles in $-\Sigma_{2d}$ are not desirable.},
the set $\bdomain$ is said to be coherent \emph{relative} to 
the pair constituted  by the  vector subspace of quadratic forms $v_{2d}(x_1,\dots,x_n)^TQv_{2d}(x_1,\dots,x_n)$ defined by the symmetric real matrices $Q$ and 
the closed\footnote{Closedness of the convex cone of SOS was proved by \cite{robinson1969some}.} convex cone of 
SOS polynomials (or equivalently the closed convex cone of polynomials defined by a positive semi-definite real-symmetric matrix).

In the multivariate case, we have seen that there are nonnegative polynomials that do not have a SOS representation. These polynomials should be in principle desirable for Alice in the ADG framework, but in BADG we do not enforce Alice to accept them.  A similar reasoning holds for the difference between A.2 and bA.2: in the chosen framework, we cannot say that Alice is irrational if she chooses  a negative gamble for which she cannot verify computationally the negativity. 
Despite the fact that in principle they should never be accepted  by Alice in the ADG framework, in a BADG we thus do not enforce this property and we only ask Alice to avoid gambles in $\Sigma^-_{2d}$.  For these reasons, BADG is a theory of bounded rationality. 


Alice may not be able to prove that  her set of desirable gambles satisfies A.2.
However, by exploiting the fact that
$$
\Sigma_{2d} \subseteq \gambles_{2d}^+ \subset \gambles^+,
$$
a BADG set $\bdomain$ that satisfies A.2 but not A.1 can (theoretically) be turned to:
\begin{enumerate}
 \item an ADG in $\gambles_{2d}$ by considering its extension $\posi(\bdomain \cup \gambles_{2d}^+)$, and also to
 \item an ADG in $\gambles$ by considering its extension  $\posi(\bdomain \cup \gambles^+)$.
\end{enumerate}

The other way round is also true. Namely:
\begin{proposition}\label{prop:fromADG2BADG}
Let $G \subseteq \gambles_{2d}$ be a  finite set of assessments, and 
assume $\domain=\posi(G \cup \gambles_{2d}^+)$ (resp.  $\domain=\posi(G \cup \gambles^+)$) is ADG. 
Then $\bdomain= \posi(G \cup \Sigma_{2d}^+ )$ is BADG.
\end{proposition}
\begin{proof}
We just verify the claim with $\gambles_{2d}^+$, the other, mutatis mutandis, being verified the same way.
Since by assumption $\domain$ is ADG, $\domain$ is a closed convex cone containing $\gambles^+$ and disjoint from the set of negative gambles. In particular $\domain$ is BADG. Therefore 
$\bdomain= \posi(G \cup \Sigma_{2d}^+ ) \subseteq \domain$ 
and it is a closed convex cone, meaning it is BADG too.
\end{proof}
These remarks are important because, as it will be shown in the next sections, they
 will allow us to use BADG as  a computable approximation of ADG.

In BADG  theory, Proposition \ref{th:noneg} is reformulated as follows.
\begin{theorem} 
\label{th:Bnoneg} Given a finite set $G \subset \gambles_{2d}$ of desirable gambles, the set $\posi(G \cup \Sigma_{2d})$  includes the gamble $f$ if and only if  there exist $\lambda_j\geq 0$ for $j=1,\dots,|G|$ such that
\begin{equation}
\label{eq:feasibil1}
f- \sum_{j=1}^{|G|}\lambda_jg_j \in \Sigma_{2d}.
\end{equation} 
\end{theorem}
\begin{proof}
Assume that $ f\in \posi(G\cup \Sigma_{2d})$ then there exist $\sigma_i\in  \Sigma_{2d}$ and $\gamma_i \geq 0$ for $i=1,\dots,m$ such that
 $ f=\sum_{j=1}^{|G|}\lambda_jg_j +\sum_{j=1}^{m}\gamma_i\sigma_i$.   This implies that $  f-\sum_{j=1}^{|G|}\lambda_jg_j$ is SOS,   proving one implication. The other implication can be proven similarly.  \end{proof}
 
 If we compare  Proposition \ref{th:noneg} and Theorem \ref{th:Bnoneg}, then we see  the  difference between ADG and BADG:
 $$
 \text{ADG: }~~~ f- \sum_{j=1}^{|G|}\lambda_jg_j \in \gambles^+,~~~~~~~~~ \text{BADG: }~~~ f- \sum_{j=1}^{|G|}\lambda_jg_j \in \Sigma_{2d},
 $$
that consists in the definition of nonnegativity or, equivalently, nonnegative gambles, i.e., the gambles that Alice shall always accept.
%
  
  Also for BADG we can consider the gamble $f=h-\lambda_0$ for some $\lambda_0 \in \mathbb{R}$ and define the concept of lower prevision.
  \begin{definition} Let  $G \subset \gambles_{2d}$ be a finite set, and let $\bdomain=\posi(G \cup \Sigma_{2d})$.
  Assume that $\bdomain$ is BADG, then the solution of the following problem
  \begin{equation}
\label{eq:lowerprebounded}
\begin{array}{l}
\sup\limits_{\lambda_0 {\in \mathbb{R},}\lambda_j\geq0} \lambda_0, ~~~~~~s.t.~~~~~h-\lambda_0- \sum\limits_{j=1}^{|G|}\lambda_jg_j \in  \Sigma_{2d},
\end{array}
\end{equation} 
is called the lower prevision of $h$ and denoted as $\underline{E}^*[h]$.
 \end{definition}
We can prove that $\bdomain=\posi(G \cup \Sigma_{2d})$  satisfies   \textit{bounded} avoiding sure loss exploiting the following result.

 \begin{proposition}
 Let us consider   $\bdomain=\posi(G \cup \Sigma_{2d})$ and the following problem:
  \begin{equation}
\label{eq:avs_badg}
\begin{array}{l}
\sup\limits_{0\leq \lambda_0\leq1,~\lambda_j\geq0} \lambda_0, ~~~~s.t.~~~ -\lambda_0- \sum\limits_{j=1}^{|G|}\lambda_jg_j \in \Sigma_{2d}.
\end{array}
\end{equation} 
 $\bdomain$ does not satisfy b.A.2, and thus incurs in a sure loss  iff the above problem has solution $\lambda_0^*=1$, and it avoids bounded sure loss (b.A.2 is satisfied) iff $\lambda^*_0=0$.
 \end{proposition}
\begin{proof}
Assume that $\bdomain$ b.A.2 is false, and thus incurs in a sure loss. This means there exists $\lambda_j\geq 0$ for $j=1,\dots,|G|$ such that
 $f:=\sum_{j=1}^{|G|}\lambda_jg_j \in \Sigma^{-}_{2d}$. Hence,  $-f$ is SOS (belongs to $ \Sigma_{2d}$) and strictly positive, yielding that we can increase $\lambda_0$ as much as we want, because we can find $\sigma\in \Sigma_{2d}$ such that $- f=\sigma +\lambda_0$ is true
(given $- f > 0$). In practice, we are exploiting the fact that for any positive scaling constant $\rho$ the following equality still holds $- f \rho=\rho\sigma +\rho\lambda_0$  and so $\lambda_0=1$.
Now assume that  there is no $\lambda_j\geq 0$ for $j=1,\dots,|G|$ such that
 $\sum_{j=1}^{|G|}\lambda_jg_j \in \Sigma^{-}_{2d}$. The only way for $-\lambda_0- \sum\limits_{j=1}^{|G|}\lambda_jg_j$ being SOS is that
 $\lambda_0=0$.
\end{proof}

\subsection{Duality for BADG}\label{subsec:badg}
We can also define the dual of a BADG. In this case,  the gambles $g$ are polynomials,  the nonnegative gambles that Alice accepts are SOS, and the negative gambles that she does not accept are the negative polynomials $g$ such that $-g$ is SOS.
Since we are dealing with a vector space, we can consider its dual space $\gambles_{2d}^\bullet$ of all linear maps $L: \gambles_{2d} \rightarrow \reals$ (linear functionals) and thus define  the dual of $\bdomain\subset \gambles_{2d}$ as the set $\left\{L \in \gambles_{2d}^\bullet:  L(g) \geq0, ~\forall g \in \bdomain\right\}$.
Since {$\gambles_{2d}$} has a basis, i.e., the monomials, if we introduce the scalars
\begin{equation}
\label{eq:linearoperator}
y_{\alpha_1\alpha_2\dots \alpha_n}:={L(x_1^{\alpha_1}x_2^{\alpha_2}\cdots x_n^{\alpha_n})} \in \reals,
\end{equation} 
where $\alpha_i \in \mathbb{N}$,  we can rewrite 
$L(g)$ for any polynomial $g$ as a function of the vector of variables $y\in \reals^{s_n(2d)}$, whose components are the real variables $y_{\alpha_1\alpha_2\dots \alpha_n}$ defined above. This means that the dual space {$\gambles_{2d}^\bullet$} is isomorphic to
$\reals^{s_n(2d)}$, and we can thence define the dual maps $(\cdot)^\bullet$ between $\gambles_{2d}$ and $\reals^{s_n(2d)}$ as follows.
\begin{definition}\label{def:dual0}
Let $\bdomain$ be a subset of $\gambles_{2d}$. Its dual is defined as 
\begin{equation}
\label{eq:dualM0}
\bdomain^\bullet=\left\{{y} \in \reals^{{s_n(2d)}}:  L(g)\geq0, ~\forall g \in \bdomain\right\},
\end{equation} 
where $L(g)$ is completely determined by ${y}$ via the definition \eqref{eq:linearoperator}.
Similarly, given a subset $\mathcal{R}$ of $\reals^{s_n(2d)}$, its dual is defined as
\begin{equation}
\label{eq:dualY0}
\mathcal{R}^\bullet=\left\{{g} \in \gambles_{2d}:  L(g)\geq0, ~\forall y \in\mathcal{R}
\right\},
\end{equation} 
\end{definition}
As before, $(\cdot)^\bullet$ is an anti-monotonic operation and its image is always a closed convex cone. Furthermore,  one has that $(\cdot)_1^\bullet{^\bullet}=(\cdot)_1$, and $(\cdot)_1 \subsetneq (\cdot)_2$ if and only if $(\cdot)_2^\bullet \subsetneq (\cdot)_1^\bullet$, whenever $(\cdot)_1$ and $(\cdot)_2$ are closed convex cones.

In what follows, we verify that, analogously to Section \ref{sec:dualityADG}, the dual $\bdomain^\bullet$ is completely characterised by a closed convex set of states. But before doing that, we have to clarify what is a state in this context.
In the previous section, we defined a nonnegative linear functional (operator) as a map that assigns nonnegative real numbers to gambles that are sure gains, that is to gambles satisfying the condition for rationality axiom A.1. In the actual bounded rationality theory, we have replaced sure gains with bounded sure gains. Hence, to define what is a state  we cannot refer to nonnegative gambles but to gambles that are SOS. This means that, consistently with axiom bA.1, from the adopted bounded perspective on rationality, states are linear operators that assign nonnegative real numbers to SOS, and that additionally preserve the unit gamble. This latter condition is equivalent to: 
$$
y_0=L(1)=1.
$$
In the aim of reducing the dual of a BADG to sets of states, we thus first provide a characterisation of nonnegative linear operators. 
In doing so, we define the matrix $M_{n,d}({y}):=L({v}_d(x_1,\dots,x_n){v}_d(x_1,\dots,x_n)^\top)$, where  the linear operator is applied component-wise.
For instance, in the case $n=1$ and $d=2$, we have that
$$
M_{1,2}({y})=L({v}_2(x_1){v}_2(x_1)^\top)=L\left(\begin{bmatrix}
 1 & x_1 & x_1^2\\
 x_1 & x_1^2 & x_1^3\\
 x_1^2 & x_1^3 & x_1^4\\
\end{bmatrix}\right)=\begin{bmatrix}
 y_0 & y_1 & y_2\\
 y_1 & y_2 & y_3\\
 y_2 & y_3 & y_4\\
\end{bmatrix}.
$$
In discussing properties of the dual space, we need the following well-known result from linear algebra:
\begin{lemma}\label{fact:TR}
For any $M \in \reals^{n\times n}$ and $v \in \reals^n$, it holds that
\begin{equation}\label{eq:matrixTR}Tr(M (v v^\top)) = Tr((v v^\top)M) =  v^\top M v.\end{equation}
\end{lemma}
\begin{proposition}\label{prop:LisTR}
Let $g \in \reals_{2d}[x_1,\dots,x_n]$ and $Q$ a real symmetric-matrix such that $g(x_1,\dots,x_n)= v^\top_{d}(x_1,\dots,x_n)Q v_{d}(x_1,\dots,x_n)$. Then 
for every $y \in \reals^{{s_n(2d)}}$, it holds that $L(g) = Tr({Q} M_{n,d}({y}))$, 
where $M_{n,d}({y})=L({v}_d(x_1,\dots,x_n){v}_d(x_1,\dots,x_n)^\top)$ and $L(g)$ is completely determined by ${y}$ via the definition \eqref{eq:linearoperator}.
\end{proposition}
\begin{proof}
By Lemma \ref{fact:TR} and linearity of $L$ and trace, we obtain that 
\begin{align}
\nonumber
L(g) & = L({v}_d(x_1,\dots,x_n)^\top {Q} {v}_d(x_1,\dots,x_n)) \\
\nonumber
 &= L( Tr( {Q} {v}_d(x_1,\dots,x_n){v}_d(x_1,\dots,x_n)^\top)) \\
\nonumber
&= Tr( {Q} L({v}_d(x_1,\dots,x_n){v}_d(x_1,\dots,x_n)^\top))\\
\nonumber
&=Tr({Q} M_{n,d}({y})),
\end{align}
 where $M_{n,d}({y})=L({v}_d(x_1,\dots,x_n){v}_d(x_1,\dots,x_n)^\top)$. 
\end{proof}

We then verify that
\begin{proposition}\label{prop:igno}
Let $\bdomain=\Sigma_{2d}$. Then its dual is 
\begin{equation}\label{eq:dualmoment}
\bdomain^\bullet=\left\{{y} \in \reals^{{s_n(2d)}}: M_{n,d}({y})\geq0\right\},
\end{equation} 
\end{proposition}
\begin{proof}
Recall that, by definition, $M_{n,d}({y})=L({v}_d(x_1,\dots,x_n){v}_d(x_1,\dots,x_n)^\top)$, and that, by Equation \eqref{eqn:SOSrepr}, any $g \in \Sigma_{2d}$ can be written as ${v}_d(x_1,\dots,x_n)^\top {Q} {v}_d(x_1,\dots,x_n)$, with ${Q}\geq0$.
Fej\'er's trace theorem \citep[Ex.2.24]{boyd2004convex} states that a matrix $M \in \reals_s^{s_1(d)\times s_1(d)}$ is $M \geq 0$ if and only if $Tr(Q M)\geq 0$, for any $Q \geq 0$.
Hence Equation \eqref{eq:dualmoment} is an immediate consequence of the following equivalences:
$$\begin{array}{ccr}
 & M_{n,d}({y})\geq0 & \\
 \iff & Tr({Q} M_{n,d}({y})), \forall Q \geq 0 & \text{(Fej\'er's trace theorem)} \\
\iff & L(g) \geq 0, \forall g \in \Sigma_{2d} & \text{(Proposition \ref{prop:LisTR})}\\
\end{array}$$
  \end{proof}
Hence a  linear operator $L$ is nonnegative if and only if  $M_{n,d}({y})\geq0$.
Obviously, if $y=0$, then $L(g)\geq 0$, for each  $ g \in \Sigma_{2d} $, and thus $\{0\}^\bullet= \gambles_{2d}$. From this observation, Proposition \ref{prop:igno} and the properties of $(\cdot)^\bullet$ we therefore immediately get:
\begin{proposition}\label{prop:Bczesc}
 It holds that 
 \begin{enumerate}
   \item $(\gambles_{2d})^\bullet=\{0\}$ and $\gambles_{2d}=(\{0\})^\bullet$;
 \item $(\Sigma_{2d})^\bullet=\mathcal{Y}^+$ and $\Sigma_{2d}=(\mathcal{Y}^+)^\bullet$,
  \end{enumerate}
  where $\ydomain^+:=\left\{{y} \in \reals^{{s_n(2d)}}: M_{n,d}({y})\geq0\right\}$
\end{proposition}

Everything is therefore ready to provide an analogous characterisation of coherence as done with Proposition \ref{prop:czesc2} but for BADG.
\begin{proposition}\label{prop:Bczesc2}
Let $\bdomain \subseteq \gambles_{2d}$ be a closed convex cone. The following claims are equivalent
 \begin{enumerate}
   \item $\bdomain$ is coherent;
     \item $\bdomain\supseteq \Sigma_{2d}$ and $\bdomain \neq \gambles_{2d}$;
   \item $(\bdomain)^\bullet\subseteq \mathcal{Y}^+$ and $(\bdomain)^\bullet \neq \{0\}$.
  \end{enumerate}
\end{proposition}
\begin{proof} (2) $\Leftrightarrow$ (3): From Proposition \ref{prop:Bczesc}, $(\bdomain)^\bullet = \{0\}$ if and only if $\bdomain = \gambles_{2d}$, and $\bdomain\supseteq\Sigma_{2d}$ if and only if $(\bdomain)^\bullet\subseteq \mathcal{Y}^+$.
\\ (1) $\Rightarrow$ (2): Assume $\bdomain$ is coherent. By bA.1 $\bdomain \supseteq \Sigma_{2d}$, and by bA.2 there is $-g \in \Sigma_{2d}$ such that $\sup g < 0$ and $g \notin \bdomain$. 
\\ (2) $\Rightarrow$ (1): Let $\Sigma_{2d} \subseteq \bdomain \subsetneq \gambles_{2d}$. 
First of all, notice that, by Proposition \ref{prop:Bczesc}, $(\bdomain)^\bullet\subseteq (\Sigma_{2d})^\bullet = \mathcal{Y}^+$. 
Now, assume that $\bdomain$ is not coherent. This means
bA.2 fails. Hence we can pick $-g \in \Sigma_{2d}$ such that $\sup g < 0$ and $g \in \bdomain$. 
Consider $y \in \mathcal{Y}^+$. It holds that $L(g) \geq 0$ if and only if $y=0$, meaning that $(\bdomain)^\bullet=\{0\}$ and therefore, by Proposition~\ref{prop:Bczesc} again, $\bdomain=\gambles_{2d}$, a contradiction.
\end{proof}

As before, we denote by $\mathscr{S}$ the set of states (here seen as a subset of ${y} \in \reals^{{s_n(2d)}}$). 
By Proposition \ref{prop:Bczesc2} and reasoning exactly as for Theorem \ref{prop:dualcharges}, we then have the following result (see for instance \citealt{lasserre2009moments}).
\begin{theorem}
\label{th:dualSOS}
The map
\[ \mathcal{C} \mapsto \mathcal{P}:=\mathcal{C}^\bullet\cap \mathscr{S}\]
is a bijection between BADGs and closed convex subsets of $\mathscr{S}$. We can therefore identify the dual of 
a BADG $\bdomain$ with
\begin{equation}
\label{eq:dualM1}
\bdomain^\bullet=\left\{{y} \in \reals^{{s_n(2d)}}:  L(g)\geq0, ~~ L(1)=1, ~~ M_{n,d}({y})\geq0, ~\forall g \in \bdomain\right\},
\end{equation} 
where $L(g)$ is completely determined by ${y}$ via the definition \eqref{eq:linearoperator}.
\end{theorem}
\begin{proof}
Let $\bdomain$ be a coherent BADG. By Proposition \ref{prop:Bczesc2}, we get that $\bdomain^\bullet$ is a closed convex cone included in $\mathcal{Y}^+$ that does not reduced to the origin.  
Thus $\mathcal{P}$ is nonempty. Preservation of closedness and convexity by finite intersections yields that $\mathcal{P}$ is closed and convex. Furthermore $\mathbb{R}_+\mathcal{P}=(\bdomain)^\bullet$, and therefore $\bdomain= (\mathbb{R}_+\mathcal{P})^\bullet$, where $\mathbb{R}_+\mathcal{P}:=\{ \lambda y : \lambda \geq 0, y \in \mathcal{P}\}$, meaning that the map is an injection.
We finally verify that the map is also a surjection. To do this, let $\mathcal{P} \subseteq \mathscr{S}$ be a non empty closed convex set of states. It holds that $\mathbb{R}_+\mathcal{P}$ is a closed convex cone included in $\mathcal{Y}^+$ different from $\{0\}$. Again by Proposition \ref{prop:Bczesc2},  we  conclude that the dual $(\mathbb{R}_+\mathcal{P})^\bullet$  is a coherent BADG and $\mathcal{P}=\mathbb{R}_+\mathcal{P}\cap \mathscr{S}= (\mathbb{R}_+\mathcal{P}){{}^\bullet{^\bullet}}\cap \mathscr{S}$. 
\end{proof}


In Section \ref{sec:unbounded gambles}, by considering  the space of all measurable gambles and identifying gambles representing sure gain with nonnegative gambles, states coincide with probability measures. Henceforth we have shown that the dual of an ADG is a closed convex set of probability  measures. In \eqref{eq:dualM1} there is no reference to probability, and thus there is no guaranty that in the bounded rationality case states correspond indeed to probabilities. However, by considering the Borel sigma-algebra on $\reals^n$, we can define  the integral 
$ \int x_1^{\alpha_1}x_2^{\alpha_2},\dots,x_n^{\alpha_n} d\mu$ with respect to the finite signed measure $\mu$. In this context, we can interpret 
$y_{\alpha_1\alpha_2\dots \alpha_n}$ as the expectation of $x_1^{\alpha_1}x_2^{\alpha_2},\dots,x_n^{\alpha_n}$ w.r.t.\ the  signed measure $\mu$.

Note that $y_0=L(1)=1$ implies that $\int 1 d\mu=1$ under this interpretation (normalization).
Therefore, we can interpret  $M_{n,d}({y})$ as a truncated moment matrix. 
However, since $\bdomain$ does not include all nonnegative gambles, we cannot conclude that
the signed measures are nonnegative or, in other words, that $\mu$ is  a probability measure. The constraint $M_{n,d}({y}) \geq 0$  is not strong enough to guarantee nonnegativity of $\mu$ (it is only a necessary condition).


In the standard theory, negative probabilities are considered a manifestation of incoherence. Here, they are a consequence of the assumption of bounded rationality.
Finally, the dual of the lower prevision problem \eqref{eq:lowerprebounded} is then given by the convex SemiDefinite Programming (SDP) problem:
\begin{equation}
 \label{eq:lowerpreboundeddual}
\inf\limits_{{y} \in \reals^{s_n(2d)}} L(h), ~~~~~s.t.~~~~~L(g)\geq 0, ~~~~~L(1)=1, ~~~~~M_{n,d}(y)\geq 0.
 \end{equation}

 \subsection{Non-SOS positive polynomials}
 What does it mean for our theory of bounded rationality that there are  positive polynomials that are not SOS?
 \\
 First, notice that, by Proposition \ref{prop:Bczesc2}, $g \in \Sigma_{2d}$ if and only if for every $y \in \ydomain^+$, it holds that $L(g) \geq 0$ (where $L$ is completely determined by ${y}$ via equation \eqref{eq:linearoperator}). This means that if $g$ is positive but not SOS, then there is $y \in \ydomain^+$ such that $L(g) < 0$. Another, equivalent, way to see this goes as follows.
 Let us assume that the polynomial $g\in \gambles_{2d}$ is positive but not SOS, its lower prevision is:
   $$
 \sup_{\lambda_0\in \reals} \lambda_0  \text{ s.t. } g-\lambda_0\in \Sigma_{2d}.
 $$
 The solution of the above problem is $\lambda_0<0$.
 Note in fact that, since $g(x)=v^T(x) Q v(x)$ is not SOS, this implies that the  matrix $Q$ is indefinite (it is not a PSD matrix).
Hence,  the only way to satisfy $g-\lambda_0\in \Sigma_{2d}$ is for  $\lambda_0<0$. 
\\
By duality, we can then prove that the problem
 $$
 \inf_{{y} \in \reals^{{s_n(2d)}}} L(g) \text{ s.t. } L(1)=1, ~M_{n,d}({y})\geq0,
 $$
 has a negative solution, i.e. $L(g)<0$. 
 \begin{example}
 Let us consider the positive non-SOS polynomial $f(x)=1+x_1^4x_2^2+x_1^2x_2^4-x_1^2x_2^2$ in $\reals_{6}[x_1,x_2]$, the basis  $v_3(x)=[1,x_1,x_2,x_1^2,x_1 x_2,x_2^2, x_1^3,x_1^2x_2,x_1x_2^2,x_2^3]^T$ and the following PSD matrix
 $M_{2,3}(y)=L(v_3(x)v_3(x)^T)$:
\begin{align}
\label{eq:darioMat}
\begin{bsmallmatrix}
 1 &y_{10} &y_{01} &y_{20} &y_{11} &y_{02} &y_{30} &y_{21} &y_{12} &y_{03}\\
y_{10} &y_{20} &y_{11} &y_{30} &y_{21} &y_{12} &y_{40} &y_{31} &y_{22} &y_{13}\\
y_{01} &y_{11} &y_{02} &y_{21} &y_{12} &y_{03} &y_{31} &y_{22} &y_{13} &y_{04}\\
y_{20} &y_{30} &y_{21} &y_{40} &y_{31} &y_{22} &y_{50} &y_{41} &y_{32} &y_{23}\\
y_{11} &y_{21} &y_{12} &y_{31} &y_{22} &y_{13} &y_{41} &y_{32} &y_{23} &y_{14}\\
y_{02} &y_{12} &y_{03} &y_{22} &y_{13} &y_{04} &y_{32} &y_{23} &y_{14} &y_{05}\\
y_{30} &y_{40} &y_{31} &y_{50} &y_{41} &y_{32} &y_{60} &y_{51} &y_{42} &y_{33}\\
y_{21} &y_{31} &y_{22} &y_{41} &y_{32} &y_{23} &y_{51} &y_{42} &y_{33} &y_{24}\\
y_{12} &y_{22} &y_{13} &y_{32} &y_{23} &y_{14} &y_{42} &y_{33} &y_{24} &y_{15}\\
y_{03} &y_{13} &y_{04} &y_{23} &y_{14} &y_{05} &y_{33} &y_{24} &y_{15} &y_{06}\\
\end{bsmallmatrix}=\begin{footnotesize}\begin{bsmallmatrix} 1 & 0 & 0 & 353 & 0 & 353 & 0 & 0 & 0 & 0\\ 0 & 353 & 0 & 0 & 0 & 0 & 249572 & 0 & 66 & 0\\ 0 & 0 & 353 & 0 & 0 & 0 & 0 & 66 & 0 & 249572\\ 353 & 0 & 0 & 249572 & 0 & 66 & 0 & 0 & 0 & 0\\ 0 & 0 & 0 & 0 & 66 & 0 & 0 & 0 & 0 & 0\\ 353 & 0 & 0 & 66 & 0 & 249572 & 0 & 0 & 0 & 0\\ 0 & 249572 & 0 & 0 & 0 & 0 & 706955894 & 0 & 17 & 0\\ 0 & 0 & 66 & 0 & 0 & 0 & 0 & 17 & 0 & 17\\ 0 & 66 & 0 & 0 & 0 & 0 & 17 & 0 & 17 & 0\\ 0 & 0 & 249572 & 0 & 0 & 0 & 0 & 17 & 0 & 706955894 \end{bsmallmatrix}\end{footnotesize}
\end{align}
Since  $L(f)=1+y_{42}+y_{24}-y_{22}$ and   $y_{22}=66, y_{24}=y_{42}=17$ in \eqref{eq:darioMat},  we have that $L(f)=-31<0$.
 The above matrix is PSD but it
is  not the truncated moment matrix of any probability measure (if it would be then $L(g)\nless 0$).
 \end{example}
How is it possible?

The cone $\ydomain^+=\left\{{y} \in \reals^{{s_n(2d)}}: M_{n,d}({y})\geq0\right\}$ includes the evaluation functionals of the polynomials.\footnote{An evaluation functional over $\gambles$ is a linear functional that evaluates each gamble $g$ at a point $\tilde{x}$.}
Evaluation functionals coincide with the rank one matrices $M$ 
 that give the evaluation of $g$ at a point $\tilde{x}$ as
$Tr(Q\,M)=g(\tilde{x})$,  for any decomposition of $g(x)=v_d^T(x) Q v_d(x)$. Such matrices have the form $M= v_d(\tilde{x})v_d(\tilde{x})^T$.
However, contrary to what happens in the standard theory of desirable gambles as described in Sections \ref{sec:TDG} and \ref{sec:dualityADG}, these matrices do not exhaust all extremes of the closed convex set $\ydomain^+$.

Said  in another way, since in the space of Borel measures the evaluation functionals are the atomic measures and since $v(\tilde{x})v(\tilde{x})^T=\int v_d(x)v_d(x)^T\delta_{\{\tilde{x}\}}dx $, there does not exist a mixture of atomic measures $\sum_{i=1}^m w_i \delta_{\{\tilde{x}^{(i)}\}}$ such that  
$$
\int v_3(x)v_3(x)^T \left(\sum_{i=1}^m w_i \delta_{\{\tilde{x}^{(i)}\}}\right)dx = \sum_{i=1}^m w_i v_3(\tilde{x}^{(i)})v_3(\tilde{x}^{(i)})^T =M_{2,3}(y). 
$$  
In \eqref{eq:darioMat}, the only way to satisfy the above equality is that some of  weights $w_i$ are negative.

Similarly, since $-L(-g)=L(g)$, we can also conclude that, for a negative gamble $g$ whose inverse $-g$ is not SOS, we have  $L(g)>0$.
 Alice may accept a negative gamble  not belonging to $\Sigma^-_{2d}$! In BADG, this is allowed because evaluating the negativity of a non SOS gamble not in $\Sigma^-_{2d}$
 may be computationally intractable.
 In the next section, we will show that/when we can use BADG as a computable approximating theory for ADG.

\section{BADG as an approximating theory for ADG}
\label{sec:badapprox}
We are going to show that we can use BADG as a computable approximating theory for ADG. Since we are dealing with unbounded gambles,
we will refer to the ADG formulation in Section \ref{sec:unbounded gambles}.
So let us consider the BADG set $\bdomain=\posi(G \cup \Sigma_{2d})$ and  the corresponding   ADG set
$\domain=\posi(G \cup \gambles^+)$ (same $G$), where $ \gambles^+$ is the set of measurable non-negative gambles. We have the following result.
\begin{theorem}
\label{th:badgasapprox}
 Assume that $\domain$ avoids sure loss (i.e. satisfies A.2) and let $f \in \gambles_{2d}$, then BADG is a conservative approximation of
 ADG theory in the sense that 
$\underline{E}^*(f) \leq  \underline{E}(f)$, where  $\underline{E}(f)$ is the coherent lower prevision of $f$ computed with respect to the set
of probabiltiy measures compatible with Alice's assessments of desirability $G$.
\end{theorem}
\begin{proof} Notice that since  $\domain$ satisfies A.2, $\bdomain$ also satisfies A.2, and hence bA.2. Now
 let $\lambda^*_0$ {be} the supremum value of $\lambda_0$ such that
 $ f-\lambda_0-\sum_{j}^{|G|} \lambda_j g_j \in \Sigma_{2d}$   and  $\lambda^{**}_0$ the supremum value 
 such that  $ f-\lambda_0-\sum_{j}^{|G|} \lambda_j g_j\geq0$. {Since the constraint $f-\lambda_0-\sum_{j}^{|G|} \lambda_j g_j \in \Sigma_{2d}$ is more demanding than $ f-\lambda_0-\sum_{j}^{|G|} \lambda_j g_j\geq0$, it  follows that $\lambda^*_0 \leq \lambda^{**}_0$.}  
\end{proof}

 The fact that  $\underline{E}[f]$ is equal to the  minimum of $f$ when $G$ is empty, i.e., Alice is in a state of full ignorance, is one of the reasons   why SOS polynomials   are used in optimisation.  In fact, $\underline{E}^*[f]$ provides a lower bound for the minimum of $f$ \citep{lasserre2009moments}.

\begin{example}[Covariance inequality]
Let us consider the case $n=2,d=1$, the matrix $M_{2,1}({y})$ is in this case
  $$
 M_{2,1}({y})=L\left(\begin{bmatrix}
  1 & x_1 & x_2\\
  x_1 & x_1^2 & x_1x_2\\
  x_2 & x_1 x_2 & x_2^2\\
 \end{bmatrix}\right)=\begin{bmatrix}
  1 & y_{10} & y_{01}\\
  y_{10} & y_{20} & y_{11}\\
  y_{01} & y_{11} & y_{02}\\
 \end{bmatrix}.
 $$
 Let us assume that Alice finds the following $8$ polynomials to be desirable:
 $$
 G=\{\pm(x_1-m_1),\pm(x_2-m_2),\pm(x_1^2-m_1^2-s_1^2),\pm(x_2^2-m_2^2-s_2^2)\}.
 $$
Since $ L(\pm(x_1-m_1))=\pm (y_{10}-m_1)$, $ L(\pm(x_2-m_2))=\pm (y_{01}-m_2)$, $ L(\pm (x_1^2-m_1^2-s_1^2))=\pm (y_{20}-m_1^2-s_1^2)$, $ L(\pm(x_2^2-m_2^2-s_2^2))=\pm (y_{02}-m_2^2-s_2^2)$ and given that
$y=[y_{10},y_{01},y_{20},y_{11},y_{02}]^\top $, we have that the  dual of the BADG $\bdomain$ is:
 \begin{equation}
 \bdomain^\bullet=\left\{y^\top \in \reals^{5}:y_{10}=m_1,y_{01}=m_2,~y_{20}=m_1^2+s_1^2,~y_{02}=m_2^2+s_2^2, ~ M_{2,1}(y)\geq0\right\}.
 \end{equation} 
Hence, it follows that
  $$
 M_{2,1}({y})=\begin{bmatrix}
  1 & y_{10} & y_{01}\\
  y_{10} & y_{20} & y_{11}\\
  y_{01} & y_{11} & y_{02}\\
 \end{bmatrix}=\begin{bmatrix}
  1 & m_1 & m_2\\
  m_1 & m_1^2+s_1^2 & y_{11}\\
  m_2 & y_{11} & m_2^2+s_2^2\\
 \end{bmatrix}.
 $$
 Assume that we aim at computing $\underline{E}^*[(x_1-m_1)(x_2-m_2)]$, $\overline{E}^*[(x_1-m_1)(x_2-m_2)]$, i.e., the lower/upper prevision
 of the gamble $(x_1-m_1)(x_2-m_2)$. Note that  $ L((x_1-m_1)(x_2-m_2))=y_{11}-m_1m_2$. From  $M_{2,1}(y)\geq0$ we can derive that
 $$
   - s_1 s_2 \leq y_{11}-m_1 m_2 \leq   s_1 s_2.
 $$
 These inequalities follow by  $det(M_{2,1}(y))\geq0$. From \eqref{eq:lowerpreboundeddual}, we thus have that $\underline{E}^*[(x_1-m_1)(x_2-m_2)]=- s_1 s_2$ and $\overline{E}^*[(x_1-m_1)(x_2-m_2)]=s_1 s_2$.
By interpreting $m_1,m_2,s_1^2,s_2^2$ as the means and variances of $x_1,x_2$ and observing that the above two inequalities can be written as
  $$
  (y_{11}-m_1 m_2)^2 \leq   s_1^2 s_2^2,
 $$
from Theorem \ref{th:badgasapprox} (this theorem holds because $\bdomain$ satisfies A.2) we can  derive that
\begin{equation}
\label{eq:CI}
E[(x_1-m_1)(x_2-m_2)]^2 \leq Var(x_1)^2\,Var(x_2)^2. 
\end{equation}
This is the well-known covariance inequality in probability theory. Observe that there exists a probability measure in BADG for which the above inequality is tight:
$\frac{1}{2}\delta_{\begin{psmallmatrix}m_1-s_1\\m_2-s_2\end{psmallmatrix}}(x)+\frac{1}{2}\delta_{\begin{psmallmatrix}m_1+s_1\\m_2+s_2\end{psmallmatrix}}(x)$,  here $\delta_{(a)}$ denotes an atomic measure (Dirac's delta) at $a$. It can in fact be verifed that this probability measure satisfies all the moment constraints:
$$
\begin{aligned}
 E[x_1]=m_1, ~E[x_2]=m_2,~E[x_1x_2]=m_1m_2+s_1^2s_2^2,~E[x_1^2]=m_1^2+s_1^2,E[x_2^2]=m_2^2+s_2^2.
\end{aligned}
$$
Hence, Theorem \ref{th:badgasapprox} is tight  in this case. However,  there are also signed measures that are compatible with Alice's assessments: 
$$
\displaystyle{\delta_{\begin{psmallmatrix}m_1-\tfrac{s_1}{\sqrt{2}}\\m_2-\tfrac{s_2}{\sqrt{2}}\end{psmallmatrix}}(x)-\delta_{\begin{psmallmatrix}m_1\\m_2\end{psmallmatrix}}(x)}+\delta_{\begin{psmallmatrix}m_1+\tfrac{s_1}{\sqrt{2}}\\m_2+\tfrac{s_1}{\sqrt{2}}\end{psmallmatrix}}(x),
$$
and that achieve the  equality in \eqref{eq:CI} but that  are not probabilities. 
 \end{example}


%

   \begin{example}
Consider the case $n=1,d=2$ and assume that 
$$
G=\{\pm(x_1-1),\pm(x_1^2-1)\}.
$$
Therefore, we have that  
$$
M_{1,2}(y)=L\left(\begin{bmatrix}
 1 & x_1 & x_1^2\\
 x_1 & x_1^2 & x_1^3\\
 x_1^2 & x_1^3 & x_1^4\\
\end{bmatrix}\right)=\begin{bmatrix}
 y_0 & y_1 & y_2\\
 y_1 & y_2 & y_3\\
 y_2 & y_3 & y_4\\
\end{bmatrix}=\begin{bmatrix}
 1 & 1 & 1\\
 1 & 1 & y_3\\
 1 & y_3 & y_4\\
\end{bmatrix}.
$$
Assume we are interested in computing the upper prevision $\overline{E}^*[x_1^4]$. From \eqref{eq:lowerpreboundeddual}, we have that
this upper prevision is equal to
\begin{equation}
 \label{eq:lowerpreboundeddualex}
\sup\limits_{{y_3,y_4} \in \reals^{2}} y_4~~~~s.t.~~M_{1,2}(y)\geq 0.
 \end{equation}
Note that the supremum is unbounded, since all matrices of the form
$$
M_{1,2}(y)=\begin{bmatrix}
 1 & 1 & 1\\
 1 & 1 & 1\\
 1 & 1 & y_4\\
\end{bmatrix}
$$
are positive semi-definite for every $y_4\geq1$.
$M_{1,2}(y)$  is  positive semi-definite, but it is not the truncated moment matrix of any probability measure. Note in fact that
$E[x_1]=E[x_1^2]=1$ would imply the probability measure to be atomic on $x_1=1$ and so  $\overline{E}[x_1^4]=1$ and, therefore, it 
cannot be that $E[x_1^4]>1$.  But (for instance for $y_4=2$) we can find an atomic 
 signed measure $1.014\delta_{1.043}+1.182\delta_{3.952}+0.004\delta_{-1.654}-0.920\delta_{3.938}-0.281\delta_{3.920}$ that
 has those moments, but it is not a probability measure (negative weights).
\end{example}
In the next section, we restrict $\Omega$ to avoid  unbounded previsions.


\section{BADG on semi-algebraic sets}
\label{sec:badgsemi}
Let us  assume that $\Omega$ is a semi-algebraic set, i.e., a  set described by polynomial inequalities
\begin{equation}
\label{eq:Omegaconstr}
 \Omega=\left\{x=[x_1,\dots,x_n]^\top \in \reals^n:  c_0(x)=1\geq0,~ c_j(x)\geq0, ~~j =1,\dots,|C| \right\},
\end{equation}
where $C=\{c_1,\dots, c_{|C|}\}$ with $c_j(x) \in \reals_{2n_{c_j}}[x]$ or $c_j(x) \in \reals_{2n_{c_j}-1}[x]$  (depending
if the polynomial has an even or odd degree).
That means that  Alice knows that $x$ belongs to the set $\Omega \subseteq \reals^n$. Note that when $C=\emptyset$, we have $\Omega=\reals^n$.
We have introduced  the redundant constraint $c_0(x)=1\geq0$ for convenience in the proofs to follow.

In ADG, the knowledge that $x$ belongs to $\Omega \subset \reals^n$  changes the cone of nonnegative gambles in $\reals^n$ from all gambles $g$ such that $g\geq0$
to all gambles $g$ such that $gI_{\Omega} \geq0$. In other words, the cone of nonnegative gambles is in this case:
\begin{equation}
\label{eq:gamblesom}
\gambles^+_{\Omega}=\{g:\reals^n\rightarrow \reals:~gI_{\Omega} \geq0\}. 
\end{equation}
Actually in ADG we do not need to change A.1 to take into account the information $x \in \Omega$, because we can define the cone of nonnegative gambles
directly in $\Omega$ 
\begin{equation}
\label{eq:gamblesom1}
\gambles^+=\{g:\Omega \rightarrow \reals:~g \geq0\}. 
\end{equation}
 To explain that, let us go back for a moment to the coin toss example  but considering the possibility space $\{Head,Tail,Side\}$.
 A gamble $g$ in this case has three components $g(Head)=g_1$,  $g(Tail)=g_2$ and $g(Side)=g_3$.
  If Alice is in a state of complete ignorance, according to A.1 she shall only accept gambles such that $g_i\geq0$ for $i=1,2,3$.
  Her set of desirable gambles is depicted in Figure \ref{fig:3d} (left),  that is the set of all nonnegative gambles in $\reals^3$.
  Assume she knows that the possibility space is actually $\Omega=\{Head,Tail\}$ ($Side$ is impossible),  according to A.1   she shall then accept all gambles
    $\{g=[g_1,g_2] \in \Reals^2:~ [g_1,g_2]\geq0 \}$ (this is the meaning of \eqref{eq:gamblesom1}), which are all the nonnegative gambles in $\reals^2$.
  Equivalently, according to \eqref{eq:gamblesom}, we can see this last cone as the 2D projection of the cone $\{g=[g_1,g_2,g_3] \in \Reals^3:~ gI_{\Omega}=[g_1,g_2,0]\geq0 \}$,
 which is   showed in Figure \ref{fig:3d} (right). Hence, in  $\reals^3$, the knowledge $\Omega=\{Head,Tail\}$  may be translated in a new definition
 of the cone of nonnegative gambles (Figure \ref{fig:3d} (right)), although this is not necessary in ADG.

  \begin{figure}
  \centering
   \includegraphics[width=5cm]{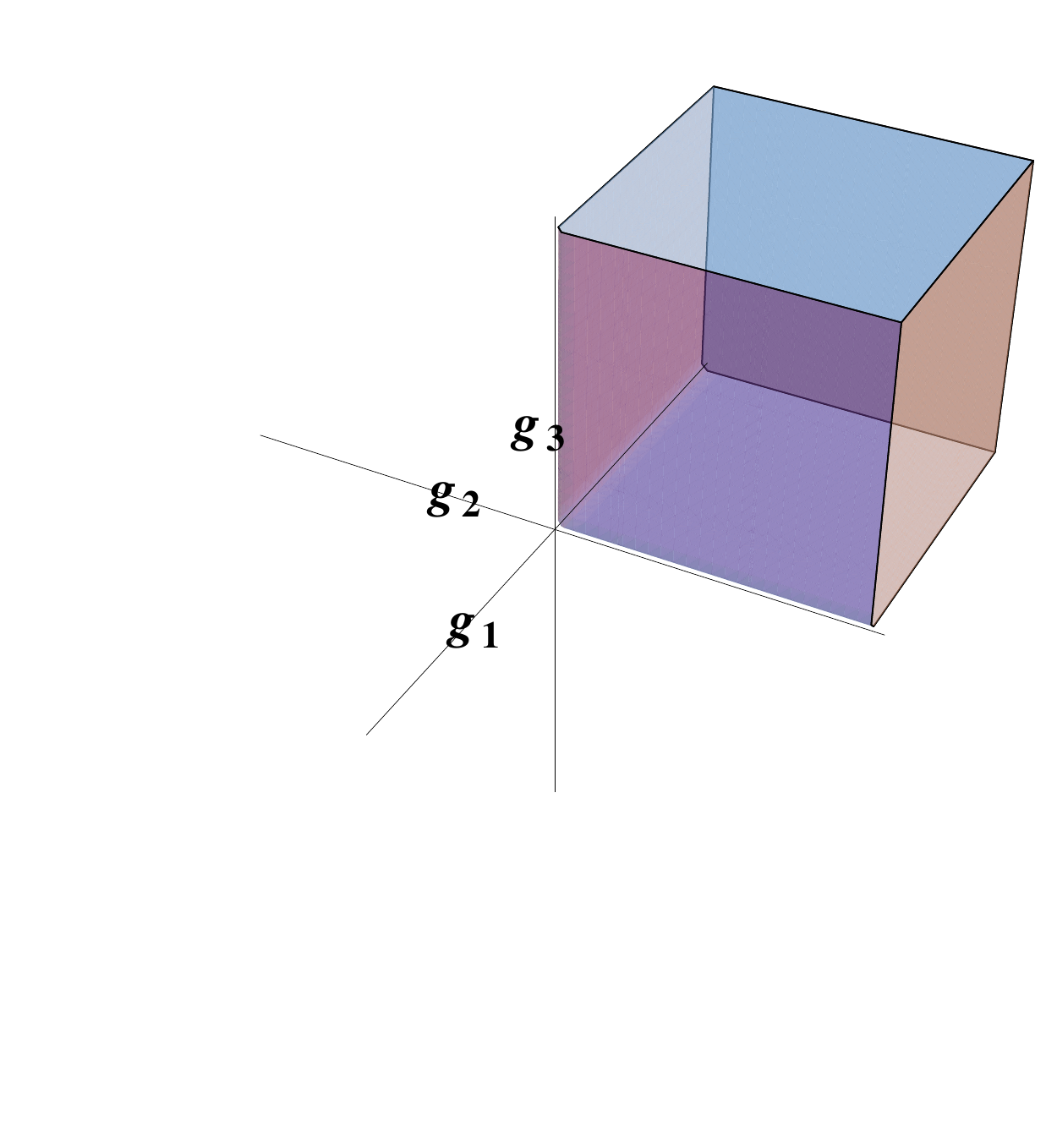} \includegraphics[width=6cm]{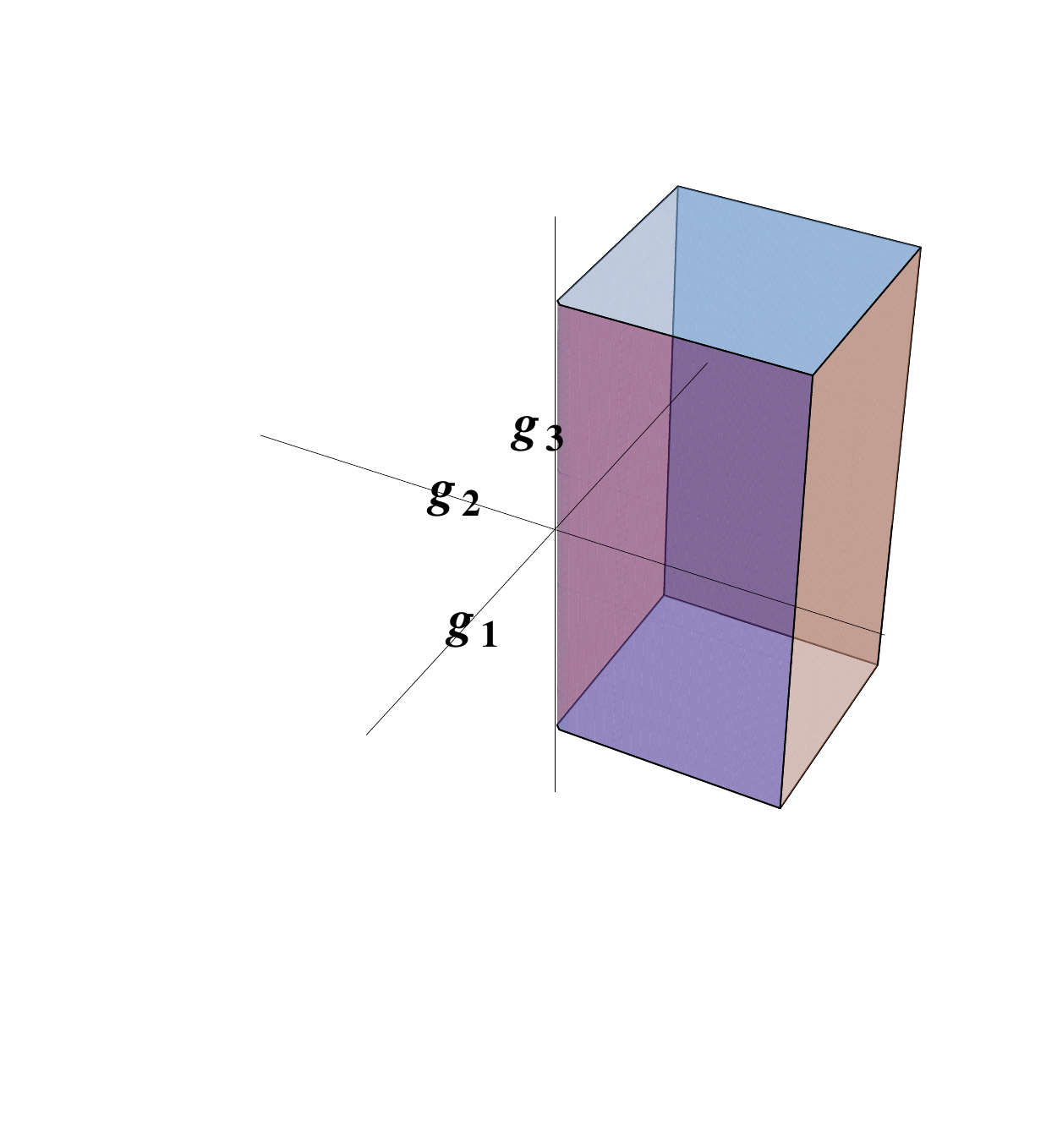}
   \caption{Cones of nonnegative gambles}
   \label{fig:3d}
  \end{figure}

In BADG, to express   the knowledge $x \in \Omega$, we cannot use \eqref{eq:gamblesom} because indicator functions are not polynomials. Similarly, we
cannot use \eqref{eq:gamblesom1}. The reason is that SOS are the computable nonnegative polynomials in $\Reals^n$ and if we restrict the domain to $\Omega$, then (in general)  we  do not know  an equivalent  class of computable nonnegative  polynomials in $\Omega$.
Hence, we need to find another way to model  $x \in \Omega$.

Let  us consider \eqref{eq:gamblesom} and notice that, for every nonnegative gamble $g:\reals^n\rightarrow \reals^+$, the gamble $g\,c_j$ is in $\gambles^+_{\Omega}$ for every $c_j$. Similarly, we have that  $g \, c_i\, c_j\in \gambles^+_{\Omega}$, $g\, c_i\, c_j\, c_k\, \in \gambles^+_{\Omega}$ and so on.
The set of gambles generated in this way forms a convex cone, 
$$
\tilde{\gambles}^+_{\Omega}:=\left\{h: h=\sum_{J \subseteq \{1,\dots,|C|\}} gc_{J}, ~~g \in \gambles^+\right\},
$$
with $c_{J}=\prod_{j \in J}c_j$, that is included in   $\gambles^+_{\Omega}$.

Since $ c_i$ are polynomials, so are $\sigma \, c_i\,$ and $\sigma \, c_i\, c_j$ and so on for any SOS $\sigma$. Moreover,
since $\sigma$ is SOS and so  nonnegative, we also know that $\sigma \, c_i\,$ ,  $\sigma\, c_i\, c_j$, $\sigma\, c_i\, c_j\, c_k,$ etc.,
are nonnegative in $\Omega$. This set forms a convex sub-cone of $\tilde{\gambles}^+_{\Omega}$, 
$$
\tilde{\tilde{\gambles}}^+_{\Omega}:=\left\{h: h=\sum_{J \subseteq \{1,\dots,|C|\}} \sigma_J c_{J}, ~~\sigma_J \in  \Sigma_{2d}\right\},
$$
and the nonnegativity of its elements can be efficiently evaluated (it reduces to verify that $\sigma_J$ is SOS) \citep{schmudgen1991thek}.


It is then natural in our theory of bounded rationality to translate the constraint $x \in \Omega$ in a computable sub-cone of the previous form.

%
%
%
%
 
We therefore give the following more general definition of BADG.

\begin{definition}
\label{def:badgOmega}
We say that  $\bdomain \subset \gambles_{2d}$ is a {\bf bounded-rationality}  coherent set of almost desirable gambles (BADG)
on the semi-algebraic set $\Omega$ in \eqref{eq:Omegaconstr},  when $d\geq \max_j{n_{c_j}}$ and $\bdomain$ satisfies A.3--A.5 (i.e. it is a closed convex cone) and:
\begin{description}
 \item[bA.1] If $ g \in \Xi_{2d} $ then $ g\in \bdomain$ (bounded accepting partial gain);
  \item[bA.2] If $ g \in \Xi^-_{2d} $ then $ g\notin \bdomain$ (bounded avoiding sure loss);
  \end{description}
 where  $\Xi_{2d}$ is defined as
 $$
 \begin{aligned}  
\Xi_{2d}&=\left\{\sigma_0 c_0 + \sum_{j=1}^{|C|} \sigma_j c_j: ~~\sigma_j\in\Sigma_{2d-2n_{c_j}}\right\}\\
 &=\left\{\sigma_0 + \sum_{j=1}^{|C|} \sigma_j c_j: ~~\sigma_0\in\Sigma_{2d},~\sigma_j\in\Sigma_{2d-2n_{c_j}}\right\}.
 \end{aligned}
$$
and $\Xi^-_{2d}$ is the interior of $-\Xi_{2d}$.
\end{definition}
%

Some remarks:
\begin{enumerate}
 \item  This is our bounded rationality approximation of  $\gambles^+_{\Omega}$. It can be noticed that the set $\Xi_{2d}$ does not include the terms 
 $\sigma \, c_i\, c_j\, c_k$ that are also nonnegative in $\Omega$. The number of these terms is  exponential in the number of polynomials that define the set $\Omega$  and, therefore, in general not suitable for computational complexity reasons. 
\item In Sections \ref{eq:putinar}, we will show that, under certain assumptions on $\Omega$, this  definition of BADG is not  conservative \citep{putinar1993positive}. 
\item Definition \ref{def:badgOmega} reduces to  Definition \ref{def:badg} when $C=\emptyset$ (so that $\Xi_{2d}=\Sigma_{2d}$).
 \item Results and Definitions in Section \ref{sec:badapprox} can be generalised accordingly
 by simply taking into account that the new set of nonnegative gambles is now $\Xi_{2d}$ 
 (before it was  $\Sigma_{2d}$).
\end{enumerate}
From now on we will use Definition \ref{def:badgOmega}  as definition of BADG. It means that Alice shall accept all polynomials of the form
$\sigma_0 + \sum_{j=1}^{|C|} \sigma_j c_j$ because they are nonnegative in $\Omega$. Again this is only a sufficient condition, since  in general there  exist 
nonnegative polynomials in $\Omega$ that  cannot be expressed as $\sigma_0 + \sum_{j=1}^{|C|} \sigma_j c_j$.

\begin{example}
 Let us consider the set 
 \begin{equation}
\label{eq:Omegaconstrex}
 \Omega=\left\{x \in \reals :  ~1-x\geq 0, ~x\geq 0 \right\},
\end{equation}
and let us assume that Alice is in a state of full ignorance. According to Definition \ref{def:badgOmega}, she shall only accept
gambles $f$ such that
$$
f=\sigma_0 + \sum_{j=1}^{|C|} \sigma_j c_j,
$$
for $\sigma_0\in\Sigma_{2d},~\sigma_j\in\Sigma_{2d-2n_{c_j}}$. Assume $f=2x + x^2$ and $d=1$, to prove that $f$ is always desirable in $\Omega$ we must  show that 
 $$
 f=2x + x^2=[1,x]Q_0[1,x]^T+ q_1 x +q_2 (1-x) ,
 $$
 with $Q_0,q_i\geq0$ ($Q_0$ is a matrix, $q_i$ are scalars).
 By equating the coefficients of the polynomials we  find the solution  $q_1=2$, $q_2=0$ and 
 $$
 Q_0=\begin{bmatrix}
     0 & 0\\
0 & 1
     \end{bmatrix}.
 $$
 Since $Q_0\geq0$ is positive semi-definite and $q_1,q_2\geq0$, this shows that $f$ is nonnegative in $\Omega$.
Instead, the polynomial  $f=\tfrac{1}{8}-x(1-x)$ cannot  be written as $[1,x]Q_0[1,x]^T+ q_1 x +q_2 (1-x) $ with $Q_0,q_i\geq0$.
This polynomial is negative  for $x=1/2$.
\end{example}

%

 \begin{example}[Markov's inequality again]
 \label{sec:markov2}
 In Example \ref{eq:markov1} we have shown how to derive Markov's inequality from ADG:
 \begin{equation}
 \label{eq:lowerpreMIagain}
\begin{array}{l}
\overline{E}(I_{\{[u,\infty)\}})=\inf\limits_{\lambda_i\in \mathbb{R}} ~~\lambda_0\\
  s.t.\\
 \lambda_0+ \lambda_{1}(x-m)  \geq I_{\{[u,\infty)\}}(x), ~~\forall x \in \Omega.\\
 \end{array}
\end{equation}
Note the presence of the indicator function that is not a polynomial.
However, the indicator is a piecewise polynomial and, therefore, the above problem can be rewritten as
 \begin{equation}
 \label{eq:lowerpreMIagain1}
\begin{array}{l}
\overline{E}(I_{\{[u,\infty)\}})=\inf\limits_{\lambda_i\in \mathbb{R}} ~~\lambda_0\\
  s.t.\\
 \lambda_0+ \lambda_{1}(x-m) -1 \geq 0, ~~\forall x \in [u,x_{max}],\\
  \lambda_0+ \lambda_{1}(x-m)  \geq 0, ~~\forall x \in [0,u).\\
 \end{array}
\end{equation}
Assume that $u \in [0,x_{max}]$, we can exploit the results of this section and rewrite the BADG formulation of the above problem as
$$
\begin{aligned}
& \inf_{\lambda_i\in \mathbb{R},\sigma_j} \lambda_0\\
 &s.t.\\
 &\lambda_0 +(x-m)\lambda_1 -1 = \sigma_0(x)+ \sigma_1(x)(x-u) + \sigma_2(x)(x_{max}-x),~~ \forall x\in \reals,\\
 & \lambda_0 +(x-m)\lambda_1 = \sigma_3(x)+ \sigma_4(x)x + \sigma_5(x)(u-x),~~ \forall x\in \reals,
\end{aligned}
$$
where $\sigma_i(x) \in \Sigma_{2(d-1)}$ for $i=1,2,4,5$ and $\sigma_i(x) \in \Sigma_{2d}$ for  $i=0,3$. 
It can be verified numerically that for $d\geq 2$ and $m<u$, the solution of the above problem is equal to $m/u$ and, therefore,
it coincides with that of ADG. For $m\geq u$ ($\lambda_0=1,\lambda_1=0,\sigma_j=0$), the infimum is $1$ same as ADG. 
\end{example}

\subsection{Duality}
We now extend the duality established in Section \ref{subsec:badg} to the case of BADG defined on semi-algebraic sets. As before, the first crucial step consists in establishing the following result.

\begin{proposition}
\label{prop:dualSOS2}
Let $\bdomain=\Xi_{2d}$. Then its dual is
\begin{equation}
\label{eq:dualM}
\bdomain^\bullet=\left\{{y} \in \reals^{{s_n(2d)}}:   ~M_{n,d}({y})\geq0,~M_{n,d-n_{c_j}}(c_j\,{y})\geq0, \forall c_j\in C\right\},
\end{equation} 
where $M_{n,r}(c\,{y}):=L(c(x_1,\dots,x_n)v_r(x_1,\dots,x_n) v_r(x_1,\dots,x_n)^\top)$.
\end{proposition}
\begin{proof}
The proof is structurally the same as the one for Proposition \ref{prop:igno}. The inclusion from right to left being easy, for the other inclusion we reason as follows. First of all, notice that 
elements of $\Xi_{2d}$ are combinations of polynomials of the form $\sigma_j(x_1,\dots,x_n) c_j(x_1,\dots,x_n)$ with $\sigma_j \in \Sigma_{2(d-n_{c_j})}$,  
Any  $\sigma_j \in \Sigma_{2(d-n_{c_j})}$ can be written as $ {v}_{d-n_{c_j}}(x_1,\dots,x_n)^\top {Q} {v}_{d-n_{c_j}}(x_1,\dots,x_n)$ {(see Eq. \eqref{eqn:SOSrepr}).} 
From Equation \eqref{eq:matrixTR},
 $ c_j(x_1,\dots,x_n) {v}_{d-n_{c_j}}(x_1,\dots,x_n)^\top {Q} {v}_{d-n_{c_j}}(x_1,\dots,x_n)$ is equal to  $Tr(Q c(x_1,\dots,x_n)  {v}_{d-n_{c_j}}(x_1,\dots,x_n){v}_{d-n_{c_j}}(x_1,\dots,x_n)^\top)$ 
 with ${Q}\geq0$. Because of linearity of $L$ and trace
\begin{align}
\nonumber
&L( Tr( {Q}\, c(x_1,\dots,x_n) {v}_{d-n_{c_j}}(x_1,\dots,x_n){v}_{d-n_{c_j}}(x_1,\dots,x_n)^\top)) \\
\nonumber
&= Tr( {Q}\, L(c(x_1,\dots,x_n) {v}_{d-n_{c_j}}(x_1,\dots,x_n){v}_{d-n_{c_j}}(x_1,\dots,x_n)^\top))\\
\nonumber
&=Tr({Q}\, M_{n,d-n_{c_j}}(c_j\,{y})),
\end{align}
 where $M_{n,d-n_{c_j}}(c_j\,{y})=L(c(x_1,\dots,x_n)\,{v}_d(x_1,\dots,x_n){v}_d(x_1,\dots,x_n)^\top)$.
  This means that $ Tr(Q\, M_{n,d-n_{c_j}}(c_j\,{y}))\geq0 ~~\forall Q\geq0$, and therefore $M_{n,d-n_{c_j}}(c_j\,{y})\geq0$ for every $c_j$.
  We conclude by considering that $c_0(x)=1$.
\end{proof}
The matrix $M_{n,r}(c\,{y})$ is called localizing matrix by \citet{lasserre2009moments}. 

As an immediate consequence of Proposition \ref{prop:dualSOS2} and the properties of $(\cdot)^\bullet$, it is then possible to verify an analogous of Proposition \ref{prop:Bczesc}. Since Proposition \ref{prop:Bczesc2} also holds for  BADG defined on semi-algebraic sets, by reasoning exactly as in Theorem \ref{th:dualSOS}, we therefore can prove the following.
\begin{theorem}
\label{th:dualSOSlocaliz}
The map
\[ \mathcal{C} \mapsto \mathcal{C}^\bullet\cap \mathscr{S}\]
is a bijection between BADGs  in  the semi-algebraic set $\Omega$  and closed convex subsets of $\mathscr{S}$. We can therefore identify the dual of 
a BADG $\bdomain$ in  the semi-algebraic set $\Omega$ with
\begin{equation}
\label{eq:dualM2}
\bdomain^\bullet=\left\{{y} \in \reals^{{s_n(2d)}}:  L(g)\geq0, ~ L(1)=1, ~M_{n,d-n_{c_j}}(c_j\,{y})\geq0, \forall c_j\in C,~ M_{n,d}({y})\geq0, ~\forall g \in \bdomain\right\},
\end{equation} 
where $L(g)$ is completely determined by ${y}$ via the definition \eqref{eq:linearoperator}.
\end{theorem}

To understand the above dual set, let us consider again the following example.

\begin{example}
 Let us consider the set 
 \begin{equation}
\label{eq:Omegaconstrola}
 \Omega=\left\{x \in \reals :  ~1-x\geq 0, ~x\geq 0 \right\}.
\end{equation}
Assume that Alice is in a state of complete ignorance and that  $d=1$. Then the dual \eqref{eq:dualM2} is

\begin{align}
 &\bdomain^\bullet=\Big\{{y_0,y_1,y_2,y_3}:\\
 &y_0=1, ~ M_{1,1}(y)=\begin{bmatrix}
             y_0 & y_1\\y_1 & y_2
            \end{bmatrix}\geq0,~~M_{1,0}(c_1 y)=y_1\geq0,,~~M_{1,0}(c_2 y)=1-y_1\geq0\Big\}.
\end{align}

By interpreting $M_{1,0}(c_1 y),M_{1,0}(c_2 y)$ as truncated moment matrices, we can see
$$
M_{1,0}(c_1 y)= L(x)=\int x d\mu\geq0 , ~~M_{1,0}(c_2 y)= L(1-x)=\int (1-x) d\mu\geq0.
$$
Hence, the assessment $x\in \Omega=[0,1]$ has been relaxed in BADG to 
$E[x]\in[0,1]$.
\end{example}

\subsection{Convergence of BADG to ADG}
\label{eq:putinar}
If we consider Theorem \ref{th:Bnoneg} then we can notice that, for fixed $G$, the set $\text{posi}(G \cup \Sigma_{2d})$ depends
on the degree $d$ of the SOS polynomials $\Sigma_{2d}$.
 By increasing $d$ we add more nonnegative gambles and, therefore,
enlarge the cone $\bdomain$.  We can then ask: what happens if we increase $d\rightarrow \infty$?

Let us assume that the semi-algebraic set $\Omega$ in \eqref{eq:Omegaconstr} is compact.
The compactness implies that polynomial gambles defined on $\Omega$ are now bounded.

\begin{proposition}[\citealt{schmudgen1991thek}]
\label{prop:schmudgen1991thek}
Let $\Omega$ be as  in \eqref{eq:Omegaconstr} and compact. If $f$ is strictly positive on $\Omega$ then  
there exist $\sigma_j\in\Sigma[x_1,\dots,x_n]$ such that
\begin{equation}
f=\sum\limits_{J \subseteq \{1,\dots,|C|\}} \sigma_{J}\, c_J,
\end{equation}
where $c_J=\prod_{j\in J} c_j$.
\end{proposition}
Since $\sigma_J$ is SOS and so  nonnegative, we  know that $\sigma_J \, c_i\,$ ,  $\sigma_J\, c_i\, c_j$, $\sigma_J\, c_i\, c_j\, c_k$ etc.
are nonnegative in $\Omega$. This set is the convex sub-cone of $\tilde{\gambles}^+_{\Omega}$ we introduced  previously in Section \ref{sec:badgsemi}.
\cite{schmudgen1991thek} proved that any strictly non-negative polynomial $f$ on $\Omega$ can be written as $\sum\limits_{J \subseteq \{1,\dots,|C|\}} \sigma_{J}\, c_J$ for some 
SOS $\sigma_{J}$. 
The problem with this result is that the sum on the right hand side has an exponential number of terms. By imposing a further assumption on $\Omega$, we can make a major simplification.

We first define the following convex cone generated by the family of polynomials $c_j(x)$ \cite[Sec 2.5]{lasserre2009moments}:
$$
\Xi=\left\{\sigma_0 + \sum_{j=1}^{|C|} \sigma_j c_j: ~~\sigma_j\in\Sigma[x_1,\dots,x_n]\right\},
$$
where this time we are not restricting the degree of the SOS polynomials.

\begin{proposition}[\citealt{putinar1993positive}]
\label{prop:Putinar}
Assume that there exists a polynomial $u \in \Xi$ such that the level set $\{x: \reals^n:~u(x)\geq0\}$ is compact.
Let $\Omega$ be as  in \eqref{eq:Omegaconstr}. If $f$ is a strictly positive polynomial in $\Omega$ then $f\in \Xi$, i.e., 
there exist $\sigma_j\in\Sigma[x_1,\dots,x_n]$ such that
\begin{equation}
f=\sigma_0 + \sum_{j=1}^{|C|} \sigma_j c_j.
\end{equation}
\end{proposition}
This is a very general and powerful proposition and shows that our Definition  \ref{def:badgOmega} of BADG in $\Omega$ is not restrictive: for any 
strictly positive polynomial $f$ on $\Omega$ there exist SOS polynomials such that $f=\sigma_0 + \sum_{j=1}^{|C|} \sigma_j c_j$.
Note that, since $\Omega$ is compact, if we know a scalar $d>0$ such that $\Omega \subset \{x: \reals^n:~||x||\leq d\}$ then we can
add the constraint $||x||\leq d$ to $\Omega$ without changing $\Omega$. With this new representation, $\Xi$ satisfies the assumption
in Proposition \ref{prop:Putinar} \cite[Sec 2.5]{lasserre2009moments}.

\begin{proposition}
\label{prop:convergence}
Given the set $G$ of gambles Alice finds  desirable, a semi-algebraic set $\Omega$ satisfying the assumption in Proposition \ref{prop:Putinar}.
 Assume that $\domain$ avoids sure loss. Then for every polynomial $f$,   BADG  converges to ADG for $d \rightarrow \infty$  in the sense that when $\underline{E}^*(f)$ is finite then
 $\underline{E}^*(f) \rightarrow  \underline{E}(f)$ (from below) \cite[Th. 4.1]{lasserre2009moments}.
\end{proposition}
However, we have already  shown, for instance in the Covariance Inequality example, that it often happens that $\underline{E}^*(f) = \underline{E}(f)$ even for  finite $d$ \citep[Sec. 4.1]{lasserre2009moments}.


\section{Updating}
\label{sec:updating}
 We assume that Alice considers an event ``indicated'' by a certain  polynomial  $h(x)\geq0$, meaning  that  Alice knows that $x$ belongs to the set $A=\{x\in \reals^n: h(x)\geq0\}$. 
In ADG we will use this information to update (condition) her set of desirable gambles based on $A \subseteq \Omega$ \citep{walley1991,couso2011}. Let $G \subseteq \gambles$ be finite, and $\domain = \posi ( G \cup \gambles^+)$. Then
$\domain_{|A}=\{g\in \gambles : gI_A \in \domain\}$, where $I_A$ is the indicator function on $A$. 
From \eqref{eq:lowerpre}, it then follows that the conditional lower prevision of a gamble $f$ is 
$$
\begin{array}{rlrl}
  &\sup_{\lambda_j\geq0, \lambda_0} \lambda_0\\
  &s.t.\\
  &  (f-\lambda_0)I_A- \sum\limits_{j=1}^{|G|}\lambda_jg_j(x) \geq 0,  ~~~\forall x \in  \Omega,\\  
  \end{array}
$$
which is equivalent to
   \begin{equation}
    \label{eq:condtingamblegamble}
\begin{array}{rlrl}
  &\sup_{\lambda_j\geq0, \lambda_0} \lambda_0\\
  &s.t.\\
  &  f-\lambda_0- \sum\limits_{j=1}^{|G|}\lambda_jg_j(x) \geq 0,  ~~~\forall x \in  A,\\
  &~~~~~~~~~~~- \sum\limits_{j=1}^{|G|}\lambda_jg_j(x)\geq0,  ~~~\forall x \notin A.  
  \end{array}
\end{equation}
By writing $\lnot A:= \Omega \setminus A$, the dual of $\domain_{|A}$ coincides with

    \begin{equation}
    \label{eq:condtingamble}
\begin{array}{rlrl}
  &\inf\limits_{\mu_1 \in \mathcal{M}(A)^+,\mu_2 \in \mathcal{M}(\neg A)^+} \int\limits_A f d\mu_1 \\
  &s.t.\\
  &  \int\limits_A  d\mu_1 =1\\
  &\int\limits_A g_j d\mu_1 + \int\limits_{\neg A} g_j d\mu_2 \geq0,  ~~~\forall j=1,\dots,|G|.  
  \end{array}
    \end{equation}

  \begin{proposition}
   Assume that $\overline{E}(I_A)>0$ then the above optimisation problem is equivalent to
    \begin{equation}
    \label{eq:condtinprob00}
\begin{array}{cl}
    \sup\limits_{\nu \in \reals} \nu  : ~~\inf\limits_{\mu \in \mathcal{M}(\Omega)^+}  & \int (f-\nu)I_A d\mu\geq 0\\
   &s.t.\\
   &\int  d\mu  =1\\
   &\int g_j d\mu  \geq0,  ~~~\forall j=1,\dots,|G|.  
  \end{array}     
    \end{equation}
  \end{proposition}
This is also called  regular extension  \citep[Appendix J]{walley1991}.

How do we do that in  the BADG framework?
In BADG we cannot completely use this information because again $\Sigma_{2d}$ does not include  indicator functions.
However, we can still exploit the information in $A$ in a weaker way as shown in the previous section.
In fact, if we know that $h(x)\geq0$, then we    also know:
$$
\begin{aligned}
 \sigma_1(x) h(x)\geq 0 & ~~~\forall x \in A,\\
 -\sigma_2(x) h(x)\geq 0 & ~~~\forall x \in \lnot A,\\
\end{aligned}
$$
for $\sigma_i \in \Sigma_{2(d-n_{h})}$,  where the degree of  $h(x)$ is $2n_{h}$ if even or $2n_{h}-1$ if odd  (so that the degree of $\sigma_i(x) h(x)$ is less  than $2d$).
Hence, a possible way to define updating in BADG is as follows.

\begin{definition}
Let $G$ be a finite subset of $\gambles_{2d}$, and $\bdomain=\posi(G \cup \Xi_{2d})$ be a set of BADG in $\Omega$. Given the event $A=\{x \in  \Omega: h(x)\geq0\}$ for some polynomial 
$h(x)$, of degree $2n_{h}$ if even or $2n_{h}-1$ if odd, then, the set $\bdomain_{|A}$ that includes all the gambles $f\in \gambles_{2d}$ such that there exist
 $\lambda_i\geq0$, with $i=1, \dots, |G|$, and $\sigma_{i0} \in \Sigma_{2d}$,  $\sigma_{ij} \in \Sigma_{2(d-n_{c_j})}$,  $\sigma_{a},\sigma_{b} \in \Sigma_{2(d-n_h)}$:
\begin{equation}
\label{eq:condition}
 f -\sum_{i=1}^{|G|} \lambda_i g_i=\sigma_{10}+ \sum_{j=1}^{|C|}\sigma_{1j} c_j  + \sigma_{a} h ~~\text{ and } ~~ -\sum_{i=1}^{|G|} \lambda_i g_i=\sigma_{20}+ \sum_{j=1}^{|C|}\sigma_{2j} c_j - \sigma_{b} h
\end{equation} 
is  called the {\bf updated set of desirable gambles} based on  $A$. 
\end{definition}
The above Definition is consistent with that in \eqref{eq:condtingamblegamble}, since the condition 
$ f -\sum_{i=1}^{|G|} \lambda_i g_i=\sigma_{10}+ \sum_{j=1}^{|C|}\sigma_{1j} c_j  + \sigma_{a} h$ is sufficient  for 
$$
f -\sum_{i=1}^{|G|} \lambda_i g_i \geq 0, ~\forall ~x \in A\subseteq \Omega.
$$
   In fact, given that  $ \sigma_{10}+ \sum_{j=1}^{|C|}\sigma_{1j} c_j  + \sigma_{a} h$ is nonnegative
 in $A\subseteq \Omega$, if we can write  $f -\sum_{i=1}^{|G|} \lambda_i g_i$ as $\sigma_{10}+ \sum_{j=1}^{|C|}\sigma_{1j} c_j  + \sigma_{a} h$ then this implies
 that $f -\sum_{i=1}^{|G|} \lambda_i g_i$  is also nonnegative in $A\subseteq \Omega$.
Similarly, the condition $-\sum_{i=1}^{|G|} \lambda_i g_i=\sigma_{20}+ \sum_{j=1}^{|C|}\sigma_{2j} c_j - \sigma_{b} h$ is sufficient 
for 
$$
-\sum_{i=1}^{|G|} \lambda_i g_i \geq 0, ~\forall ~x \in \neg A.
$$
Observe that, in the state of full ignorance,  since $G$ is empty, there is only one constraint $f  =\sigma_0+\sum_{j=1}^{|C|}\sigma_{1j} c_j+ \sigma_a h$.

\begin{theorem}
Assume that $\bdomain_{|A}$ is BADG in $\Omega$. Then it holds that
$$
\begin{aligned}
\bdomain_{|A}^\bullet=&
\Big\{ y \in \reals^{s_n(d)}: 
  \exists  z \in \reals^{s_n(d)} \text{ such that }\\
 & M_{n,d}({y}),M_{n,d-n_{c_j}}(c_j{y}), M_{n,d}({z}),M_{n,d-n_{c_j}}(c_j{z})\geq0 ~\forall c_j,\\
 & M_{n,d-n_h}(hy),M_{n,d-n_h}(-hz)\geq0, ~ L_y(1)=1,\\
 & L_y(g)+L_z(g)\geq0, ~~\forall g=1,\dots,|G|\Big\}.
\end{aligned}
$$
\end{theorem}
\begin{proof}
The argument of the proof is similar to that of Proposition \ref{prop:dualSOS2}. Note in fact that to define the dual of
$ f -\sum_{i=1}^{|G|} \lambda_i g_i=\sigma_{10}+ \sum_{j=1}^{|C|}\sigma_{1j} c_j  + \sigma_{a} h$ we can exploit Proposition \ref{prop:dualSOS2} 
and account for the presence of the additional constraint $h\geq0$. The variable $z$ is introduced to define the dual of
$-\sum_{i=1}^{|G|} \lambda_i g_i=\sigma_{20}+ \sum_{j=1}^{|C|}\sigma_{2j} c_j - \sigma_{b} h$.
The constraint  $ L_y(g)+L_z(g)\geq0$ connects the two duals and arises due to the presence of the term $-\sum_{i=1}^{|G|} \lambda_i g_i$
in both equalities \eqref{eq:condition}.
\end{proof}

To understand the above dual set, we can compare it with 
\eqref{eq:condtingamble}. The vector $y$  has the same role of  $\mu_1$ and $z$ that of $\mu_2$.
The constraints $M_{n,d}({y}),M_{n,d}({z})$ are the  bounded rationality  analogous of $\mu_1,\mu_2 \in \mathcal{M}^+(\reals^n)$.
The constraints $M_{n,d-n_{c_j}}(c_j{y}),M_{n,d-n_{c_j}}(c_j{z})\geq0$ are the  bounded rationality analogous of the  support constraints  $\mu_1 \in \mathcal{M}^+(\Omega)$
and $\mu_2 \in \mathcal{M}^+(\Omega)$. The constraints $M_{n,d-n_h}(hy),M_{n,d-n_h}(-hz)\geq0$ are the bounded rationality representation of the constraints $\mu_1 \in \mathcal{M}^+(A)$ and $\mu_2 \in \mathcal{M}^+(\neg  A)$. Finally, $ L_y(g)+L_z(g)\geq0$ is equivalent to $\int_A g_j d\mu_1 + \int_{\neg  A} g_j d\mu_2 \geq0$.


\begin{theorem} Let $G$ be a finite subset of $\gambles_{2d}$, and $A=\{x \in  \Omega: h(x)\geq0\}$.  
 Assume that $\domain=\posi(G \cup \gambles^+)$ avoids sure loss and let $f \in \gambles_{2d}$. Then we have that
 $\underline{E}_{\bdomain_{|A}}(f) \leq  \underline{E}_{\domain_{|A}}(f)$ where $\bdomain=\posi(G \cup \Xi_{2d})$.
\end{theorem}
\begin{proof}
From the definition  of conditioning for ADG we aim to find the supremum $\lambda_0$ such that
$ (f-\lambda_0)I_{A} - \sum_{j=1}^{|G|}\lambda_jg_j(x) \geq 0 ~~~\forall x \in \reals^n$.
It can be rewritten as the two constraints on the left and relaxed to the constraints on the right:
$$
\begin{aligned}
  - \sum_{j=1}^{|G|}\lambda_jg_j(x)&\geq0  ~\forall x \in \lnot A,   & - \sum_{j=1}^{|G|}\lambda_jg_j(x) =\sigma_{20}+ \sum_{j=1}^{|C|}\sigma_{2j} c_j  - \sigma_{b} h,\vspace{1mm}\\
   f-\lambda_0- \sum_{j=1}^{|G|}\lambda_jg_j(x) &\geq 0  ~\forall x \in  A,  & f-\lambda_0- \sum_{j=1}^{|G|}\lambda_jg_j(x) =\sigma_{10}+ \sum_{j=1}^{|C|}\sigma_{1j} c_j  + \sigma_{a} h.
\end{aligned}
$$
where the equalities on the right must hold $~\forall x \in \reals^n$.
 \end{proof}
 
In case the set $A$ is defined by several polynomial constraints \\ $A=\{h_1(x)\geq0,\dots,h_{|A|}(x)\geq0\}$, we cannot use \eqref{eq:condition} because in general we cannot
write $\neg  A$ as a single polynomial constraint. However, we can relax  \eqref{eq:condition} to:
\begin{equation}
\label{eq:condition1}
 f -\sum_{i=1}^{|G|} \lambda_i g_i=\sigma_{10}+ \sum_{j=1}^{|C|}\sigma_{1j} c_j  + \sum_{i=1}^{|A|} \sigma_{ai} h_i ~~\text{ and } ~~ -\sum_{i=1}^{|G|} \lambda_i g_i=\sigma_{20}+ \sum_{j=1}^{|C|}\sigma_{2j} c_j,
\end{equation} 
which is a conservative approximation. It can actually be proven   that the complement of a semi-algebraic set is the union of semi-algebraic sets \citep{tarski1951decision,seidenberg1954new}. Hence, the exact way to consider the constraints $\neg  A$ is to translate them in a bunch of SDP problems.

\section{Case study: European Options}\label{sec:numex}

\begin{table}[ht]
\scalebox{0.47}{
\begin{tabular}{lrrrrrrrrrrrrrrrrrrrrrr}
 \\\textbf{ask} &    53.8 &    49.5 &    45.2 &    41.1 &    29.3 &    25.7 &    22.3 &    19.1 &    16.2 &    13.6 &    11.3 &     9.2 &     7.5 &     6.1 &     4.9 &     3.9 &     3.2 &     2.15 &     1.55 &     1.15 &     0.90 &     0.3 \\\textbf{bid} &    53.3 &    49 &    44.8 &    40.6 &    28.9 &    25.3 &    21.9 &    18.7 &    15.9 &    13.3 &    11 &     8.9 &     7.2 &     5.8 &     4.6 &     3.7 &     3 &     2&     1.40 &     1&     0.75 &     0.2 \\\hline
 \textbf{strike}  &  2490 &  2495 &  2500 &  2505 &  2520 &  2525 &  2530 &  2535 &  2540 &  2545 &  2550 &  2555 &  2560 &  2565 &  2570 &  2575 &  2580 &  2590&  2600&  2610&  2620&  2675 
\end{tabular}
}
\caption{Ask and bid price for a call option on the S\&P500 index: maturity 30days, quote day 2017-10-03.}
\label{tab:1}
\end{table}

As an example of application of BADG, we consider a problem from finance. An European call option on an underlying security  with strike $k$ and maturity $T$ gives the holder the option of buying the underlying security  at price $k$ at time $T$. If the price $S_T$ is more than $k$, then the holder will exercise the option and make a profit of $S_T - k$.
Conversely, if it is less than $k$, the holder will not exercise and does not make a profit. Thus, the payoff of this option is $\max(S_T - k, 0)$. Since options are traded, a key problem in financial economics is to determine the belief of the market about the future value of $S_T$ from the ask and bid\footnote{The bid price is the max price that a buyer is willing to pay for a security. The ask price is the min price that a seller is willing to receive.} prices of these options.
Table \ref{tab:1} shows the ask and bid price for 22 call options on the S\&P500 index. 
What does the first column of the table mean? It means that ``the market'' believes that the gambles
$\max(S_T-2490,0)-53.3$ and $53.8-\max(S_T-2490,0)$ are desirable, since there exists someone that is 
willing to sell the option $\max(S_T-2490,0)$ at price $53.8$ and to buy it at price $53.3$.
As inference, we aim to compute the 
market's selling and buying price  for the gamble $f=I_{\{[c,\infty)\}}(S_T)$ for some $c \in \reals$.\\
In this case, the set of desirable gambles includes 44 gambles:
$$
\begin{aligned}
G=\{\max(S_T-2490,0)-53.3,53.8-\max(S_T-2490,0),\dots,\\
\max(S_T-2675,0)-0.2,0.3-\max(S_T-2675,0)\}.
\end{aligned}
$$
Note that, for simplicity, we have assumed that the discount factor is one.\footnote{The discount factor is the factor by which a future cash flow must be multiplied in order to obtain the present value.}
Moreover, observe that the gambles in $G$ and $f$ are piecewise polynomials. 
We aim to apply BADG to solve this problem by exploiting the same trick used in Example \ref{sec:markov2}.
Consider for instance the case $G$ only includes $\max(S_T-2490,0)-53.3,53.8-\max(S_T-2490,0),\max(S_T-2495,0)-49,49.5-\max(S_T-2495,0)$ and $c=2490$, then the lower prevision of $f$ can be computed in BADG as:
  \begin{equation}
\label{eq:lowerprea}
{\begin{array}{l}
\sup\limits_{\lambda_0 \in \mathbb{R},\lambda_j\geq0} \lambda_0\\
~~~~s.t.~~~\\
f-\lambda_0-53.8 \lambda_1+53.3 \lambda_2-49.5 \lambda_3+49 \lambda_4=\sigma_0(S_T) +(2490-S_T)\sigma_3(S_T),\vspace{2mm}\\
f-\lambda_0+(S_T-2490) \lambda_1-(S_T-2490) \lambda_2-53.8 \lambda_1+53.3 \lambda_2-49.5 \lambda_3+49 \lambda_4\\
=\sigma_1(S_T) +(S_T-2490) \sigma_4(S_T)+(2495-S_T)\sigma_5(S_T),\vspace{2mm}\\
f-\lambda_0+(S_T-2490) \lambda_1-(S_T-2490) \lambda_2+(S_T-2495) \lambda_3-(S_T-2495) \lambda_4\\
-53.8 \lambda_1+53.3 \lambda_2-49.5 \lambda_3+49 \lambda_4=\sigma_2(S_T) +(S_T-2495)\sigma_6(S_T),\\
\end{array}}
\end{equation}
which, exploiting the definition of $f$, is equal to
  \begin{equation}
\label{eq:lowerpreb}
{\begin{array}{l}
\sup\limits_{\lambda_0 \in \mathbb{R},\lambda_j\geq0} \lambda_0\\
~~~~s.t.~~~\\
-\lambda_0-53.8 \lambda_1+53.3 \lambda_2-49.5 \lambda_3+49 \lambda_4=\sigma_0(S_T) +(2490-S_T)\sigma_3(S_T),\vspace{2mm}\\
1-\lambda_0+(S_T-2490) \lambda_1-(S_T-2490) \lambda_2-53.8 \lambda_1+53.3 \lambda_2-49.5 \lambda_3+49 \lambda_4\\
=\sigma_1(S_T) +(S_T-2490) \sigma_4(S_T)+(2495-S_T)\sigma_5(S_T),\vspace{2mm}\\
1-\lambda_0+(S_T-2490) \lambda_1-(S_T-2490) \lambda_2+(S_T-2495) \lambda_3-(S_T-2495) \lambda_4\\
-53.8 \lambda_1+53.3 \lambda_2-49.5 \lambda_3+49 \lambda_4=\sigma_2(S_T) +(S_T-2495)\sigma_6(S_T),\\
\end{array}}
\end{equation}
with $\sigma_i(S_T) \in \Sigma_{0}$ for $i=3,4,5,6$ and $\sigma_i(S_T) \in \Sigma_{2}$ for $i=0,1,2$.
This approach can be generalised to all 44 gambles and allows us to deal with piecewise  polynomials.  \\
The application of SOS polynomials to European option pricing was first proposed by \cite{lasserre2006pricing}. The authors  consider 
the problem of pricing an option given information (moments) on the probability density function of $S_T$. Here, we are considering the inverse problem and we are
interested in studying it from a desirable gambles point of view and, in particular, to investigate the effect of the updating in the inference.\\
In particular, for this example, the BADG lower and upper previsions of $f$  are shown in Figure \ref{fig:1}. It is worth noticing that they coincide with those computed using ADG  -- we have verified
it numerically by discretising $S_T$ and solving a linear programming problem. Note that  the discretisation approach can only be used when the number of variables is small and, in any case, provides only an inner approximation of the lower and upper previsions. However, since in this case BADG and ADG coincide,  we can refer to these lower and upper previsions as the lower and upper probabilities
of the event $S_T>c$.\\
Assume that we aim to update our inference given the information ``$S_T\geq 2540$ is true'', meaning  that  Alice knows that $x$ belongs to the set $A=\{S_T\in \reals: S_T-2540\geq0\}$. 
We can apply the approach discussed in Section \ref{sec:updating} and compute an updated set of desirable gambles. The corresponding lower and upper probabilities
for the event $S_T>c$ are shown in  Figure \ref{fig:2} (right, blue) together with the previous lower and upper  probability for comparison.\\
Options' data includes other information apart from bid and ask prices, such as trading volume for the day. We can use such information for updating, for example by using the trading volume to build a weighting function across the strikes. An example of weighting function $W(S_T)$ is shown in Figure \ref{fig:2} (left). We can then compute an updated BADG by replacing $f(S_T)-\lambda_0$ in \eqref{eq:lowerprea} with $(f(S_T)-\lambda_0)W(S_T)$. This is another way of defining an updating rule in BADG that is similar to updating with probability density functions in standard probability.
The updated lower probability is shown in Figure \ref{fig:2} (right, green).

\begin{figure}[!htp]
\centering
 \includegraphics[width=8cm]{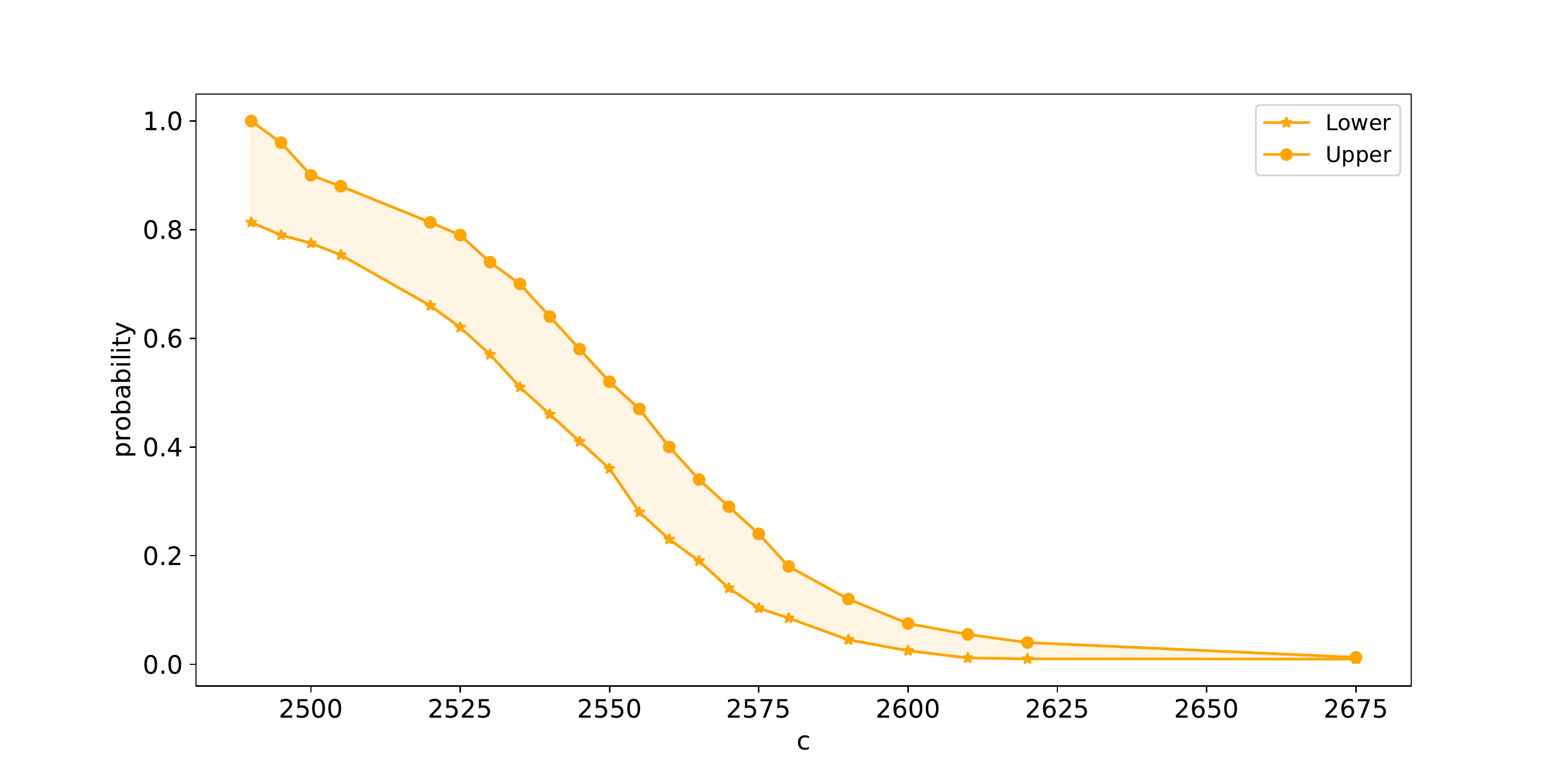}
 \caption{Lower and upper probability that $S_T>c$ for BADG.}
 \label{fig:1}
\end{figure}

\begin{figure*}[t!]
    \centering
    \begin{subfigure}[t]{0.5\textwidth}
        \centering
        \includegraphics[width=6cm]{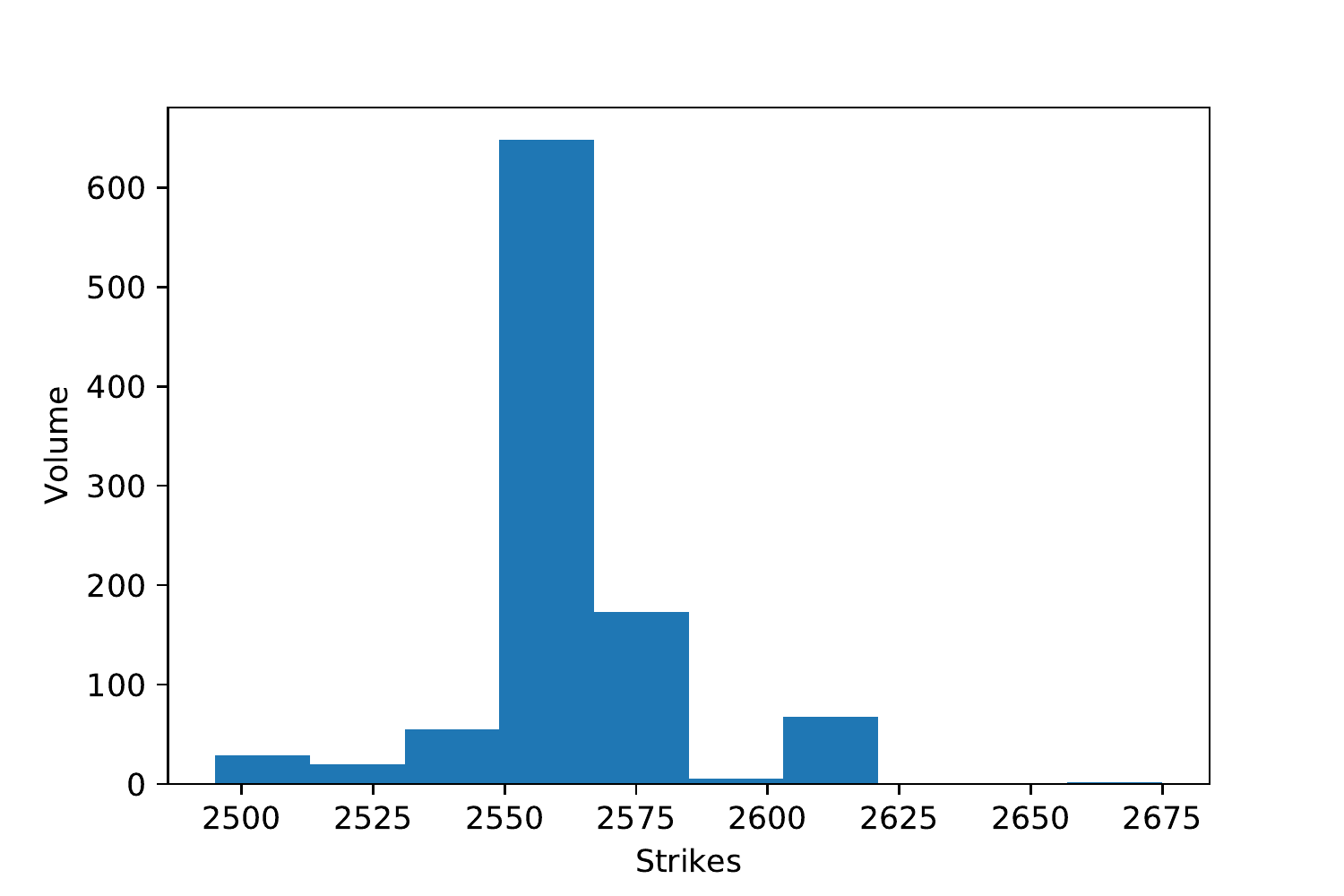}
        \caption{Weighting function}
    \end{subfigure}%
    ~ 
    \begin{subfigure}[t]{0.5\textwidth}
        \centering
        \includegraphics[width=8cm]{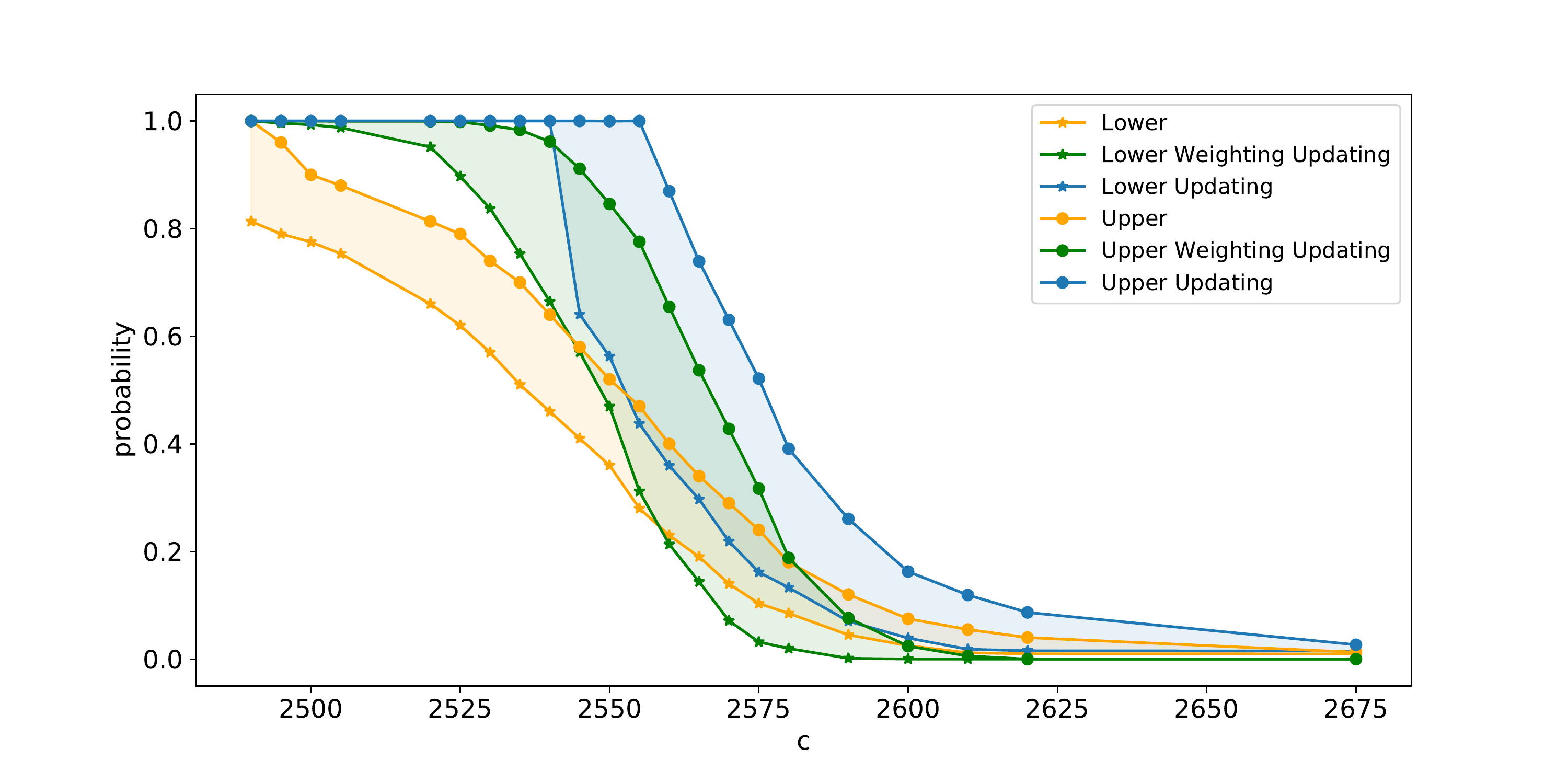}
        \caption{Lower/Upper and updated lower/upper probabilities that $S_T>c$ for BADG.}
    \end{subfigure}
    \caption{Updating}
     \label{fig:2}
\end{figure*}

\section*{Acknowledgement}
 We   thank the anonymous reviewers for their careful reading of the manuscript. Their 
  insightful comments and suggestions have greatly helped improve and clarify this work. 

 This work was partially supported  by the Swiss NRP 75 Big Data grant no. 407540-167199.


\bibliography{biblio}

\begin{thebibliography}{27}
\expandafter\ifx\csname natexlab\endcsname\relax\def\natexlab#1{#1}\fi
\providecommand{\url}[1]{\texttt{#1}}
\providecommand{\href}[2]{#2}
\providecommand{\path}[1]{#1}
\providecommand{\DOIprefix}{doi:}
\providecommand{\ArXivprefix}{arXiv:}
\providecommand{\URLprefix}{URL: }
\providecommand{\Pubmedprefix}{pmid:}
\providecommand{\doi}[1]{\href{http://dx.doi.org/#1}{\path{#1}}}
\providecommand{\Pubmed}[1]{\href{pmid:#1}{\path{#1}}}
\providecommand{\bibinfo}[2]{#2}
\ifx\xfnm\relax \def\xfnm[#1]{\unskip,\space#1}\fi
\bibitem[{Aliprantis \& Border(2007)}]{aliprantisborder}
\bibinfo{author}{Aliprantis, C.}, \& \bibinfo{author}{Border, K.}
  (\bibinfo{year}{2007}).
\newblock {\it \bibinfo{title}{Infinite Dimensional Analysis: A Hitchhiker's
  Guide}\/}.
\newblock \bibinfo{publisher}{Springer}.
\bibitem[{Benavoli et~al.(2017{\natexlab{a}})Benavoli, Facchini, Piga \&
  Zaffalon}]{Benavoli2017b}
\bibinfo{author}{Benavoli, A.}, \bibinfo{author}{Facchini, A.},
  \bibinfo{author}{Piga, D.}, \& \bibinfo{author}{Zaffalon, M.}
  (\bibinfo{year}{2017}{\natexlab{a}}).
\newblock \bibinfo{title}{Sos for bounded rationality}.
\newblock In {\it \bibinfo{booktitle}{Proc. ISIPTA'17 Int. Symposium on
  Imprecise Probability: Theories and Applications,}\/} (pp.
  \bibinfo{pages}{1--12}).
\newblock \bibinfo{publisher}{PJMLR}.
\bibitem[{Benavoli et~al.(2017{\natexlab{b}})Benavoli, Facchini, Zaffalon \&
  Vicente-Pérez}]{pmlr-v62-benavoli17b}
\bibinfo{author}{Benavoli, A.}, \bibinfo{author}{Facchini, A.},
  \bibinfo{author}{Zaffalon, M.}, \& \bibinfo{author}{Vicente-Pérez, J.}
  (\bibinfo{year}{2017}{\natexlab{b}}).
\newblock \bibinfo{title}{A polarity theory for sets of desirable gambles}.
\newblock In \bibinfo{editor}{A.~Antonucci}, \bibinfo{editor}{G.~Corani},
  \bibinfo{editor}{I.~Couso}, \& \bibinfo{editor}{S.~Destercke} (Eds.), {\it
  \bibinfo{booktitle}{Proceedings of the Tenth International Symposium on
  Imprecise Probability: Theories and Applications}\/} (pp.
  \bibinfo{pages}{37--48}).
\newblock \bibinfo{publisher}{PMLR} volume~\bibinfo{volume}{62} of {\it
  \bibinfo{series}{Proceedings of Machine Learning Research}\/}.
\bibitem[{Benavoli \& Piga(2016)}]{benavoli2016a}
\bibinfo{author}{Benavoli, A.}, \& \bibinfo{author}{Piga, D.}
  (\bibinfo{year}{2016}).
\newblock \bibinfo{title}{A probabilistic interpretation of set-membership
  filtering: Application to polynomial systems through polytopic bounding}.
\newblock {\it \bibinfo{journal}{Automatica}\/},  {\it \bibinfo{volume}{70}\/},
  \bibinfo{pages}{158 -- 172}.
\bibitem[{Boyd \& Vandenberghe(2004)}]{boyd2004convex}
\bibinfo{author}{Boyd, S.}, \& \bibinfo{author}{Vandenberghe, L.}
  (\bibinfo{year}{2004}).
\newblock {\it \bibinfo{title}{Convex optimization}\/}.
\newblock \bibinfo{publisher}{Cambridge University Press}.
\bibitem[{de~Cooman \& Quaeghebeur(2012)}]{de2012exchangeability}
\bibinfo{author}{de~Cooman, G.}, \& \bibinfo{author}{Quaeghebeur, E.}
  (\bibinfo{year}{2012}).
\newblock \bibinfo{title}{Exchangeability and sets of desirable gambles}.
\newblock {\it \bibinfo{journal}{International Journal of Approximate
  Reasoning}\/},  {\it \bibinfo{volume}{53}\/}, \bibinfo{pages}{363--395}.
\bibitem[{Couso \& Moral(2011)}]{couso2011}
\bibinfo{author}{Couso, I.}, \& \bibinfo{author}{Moral, S.}
  (\bibinfo{year}{2011}).
\newblock \bibinfo{title}{Sets of desirable gambles: Conditioning,
  representation, and precise probabilities}.
\newblock {\it \bibinfo{journal}{International Journal of Approximate
  Reasoning}\/},  {\it \bibinfo{volume}{52}\/}, \bibinfo{pages}{1034--1055}.
\bibitem[{{d}e Finetti(1937)}]{finetti1937}
\bibinfo{author}{{d}e Finetti, B.} (\bibinfo{year}{1937}).
\newblock \bibinfo{title}{La pr\'evision: ses lois logiques, ses sources
  subjectives}.
\newblock {\it \bibinfo{journal}{Annales de l'Institut Henri Poincar\'e}\/},
  {\it \bibinfo{volume}{7}\/}, \bibinfo{pages}{1--68}.
\bibitem[{Hilbert(1888)}]{hilbert1888darstellung}
\bibinfo{author}{Hilbert, D.} (\bibinfo{year}{1888}).
\newblock \bibinfo{title}{{\"U}ber die darstellung definiter formen als summe
  von formenquadraten}.
\newblock {\it \bibinfo{journal}{Mathematische Annalen}\/},  {\it
  \bibinfo{volume}{32}\/}, \bibinfo{pages}{342--350}.
\bibitem[{Lasserre(2009)}]{lasserre2009moments}
\bibinfo{author}{Lasserre, J.~B.} (\bibinfo{year}{2009}).
\newblock {\it \bibinfo{title}{Moments, positive polynomials and their
  applications}\/} volume~\bibinfo{volume}{1}.
\newblock \bibinfo{publisher}{World Scientific}.
\bibitem[{Lasserre et~al.(2006)Lasserre, Prieto-Rumeau \&
  Zervos}]{lasserre2006pricing}
\bibinfo{author}{Lasserre, J.-B.}, \bibinfo{author}{Prieto-Rumeau, T.}, \&
  \bibinfo{author}{Zervos, M.} (\bibinfo{year}{2006}).
\newblock \bibinfo{title}{Pricing a class of exotic options via moments and sdp
  relaxations}.
\newblock {\it \bibinfo{journal}{Mathematical Finance}\/},  {\it
  \bibinfo{volume}{16}\/}, \bibinfo{pages}{469--494}.
\bibitem[{Miranda(2008)}]{miranda2008a}
\bibinfo{author}{Miranda, E.} (\bibinfo{year}{2008}).
\newblock \bibinfo{title}{A survey of the theory of coherent lower previsions}.
\newblock {\it \bibinfo{journal}{International Journal of Approximate
  Reasoning}\/},  {\it \bibinfo{volume}{48}\/}, \bibinfo{pages}{628--658}.
\bibitem[{Miranda \& Zaffalon(2010)}]{miranda2010c}
\bibinfo{author}{Miranda, E.}, \& \bibinfo{author}{Zaffalon, M.}
  (\bibinfo{year}{2010}).
\newblock \bibinfo{title}{Notes on desirability and conditional lower
  previsions}.
\newblock {\it \bibinfo{journal}{Annals of Mathematics and Artificial
  Intelligence}\/},  {\it \bibinfo{volume}{60}\/}, \bibinfo{pages}{251--309}.
\bibitem[{Pelessoni \& Vicig(2016)}]{pelessoni20162}
\bibinfo{author}{Pelessoni, R.}, \& \bibinfo{author}{Vicig, P.}
  (\bibinfo{year}{2016}).
\newblock \bibinfo{title}{2-coherent and 2-convex conditional lower
  previsions}.
\newblock {\it \bibinfo{journal}{International Journal of Approximate
  Reasoning}\/},  {\it \bibinfo{volume}{77}\/}, \bibinfo{pages}{66--86}.
\bibitem[{Piga \& Benavoli(2017)}]{benavoli2016c}
\bibinfo{author}{Piga, D.}, \& \bibinfo{author}{Benavoli, A.}
  (\bibinfo{year}{2017}).
\newblock \bibinfo{title}{A unified framework for deterministic and
  probabilistic d-stability analysis of uncertain polynomial matrices}.
\newblock {\it \bibinfo{journal}{Automatic Control, IEEE Transactions on}\/},
  {\it \bibinfo{volume}{62}\/}, \bibinfo{pages}{5437--5444}.
\bibitem[{Putinar(1993)}]{putinar1993positive}
\bibinfo{author}{Putinar, M.} (\bibinfo{year}{1993}).
\newblock \bibinfo{title}{Positive polynomials on compact semi-algebraic sets}.
\newblock {\it \bibinfo{journal}{Indiana University Mathematics Journal}\/},
  {\it \bibinfo{volume}{42}\/}, \bibinfo{pages}{969--984}.
\bibitem[{Robinson(1969)}]{robinson1969some}
\bibinfo{author}{Robinson, R.~M.} (\bibinfo{year}{1969}).
\newblock \bibinfo{title}{Some definite polynomials which are not sums of
  squares of real polynomials}.
\newblock In {\it \bibinfo{booktitle}{Notices of the American Mathematical
  Society}\/} (p. \bibinfo{pages}{554}).
\newblock volume~\bibinfo{volume}{16}.
\bibitem[{Schervish et~al.(2000)Schervish, Seidenfeld \&
  Kadane}]{schervish2000sets}
\bibinfo{author}{Schervish, M.~J.}, \bibinfo{author}{Seidenfeld, T.}, \&
  \bibinfo{author}{Kadane, J.~B.} (\bibinfo{year}{2000}).
\newblock \bibinfo{title}{How sets of coherent probabilities may serve as
  models for degrees of incoherence}.
\newblock {\it \bibinfo{journal}{International Journal of Uncertainty,
  Fuzziness and Knowledge-Based Systems}\/},  {\it \bibinfo{volume}{8}\/},
  \bibinfo{pages}{347--355}.
\bibitem[{Schm{\"u}dgen(1991)}]{schmudgen1991thek}
\bibinfo{author}{Schm{\"u}dgen, K.} (\bibinfo{year}{1991}).
\newblock \bibinfo{title}{The k-moment problem for compact semi-algebraic
  sets}.
\newblock {\it \bibinfo{journal}{Mathematische Annalen}\/},  {\it
  \bibinfo{volume}{289}\/}, \bibinfo{pages}{203--206}.
\bibitem[{Seidenberg(1954)}]{seidenberg1954new}
\bibinfo{author}{Seidenberg, A.} (\bibinfo{year}{1954}).
\newblock \bibinfo{title}{A new decision method for elementary algebra}.
\newblock {\it \bibinfo{journal}{Annals of Mathematics}\/},  (pp.
  \bibinfo{pages}{365--374}).
\bibitem[{Seidenfeld et~al.(1990)Seidenfeld, Schervish \&
  Kadane}]{seidenfeld1990}
\bibinfo{author}{Seidenfeld, T.}, \bibinfo{author}{Schervish, M.~J.}, \&
  \bibinfo{author}{Kadane, J.~B.} (\bibinfo{year}{1990}).
\newblock \bibinfo{title}{Decisions without ordering}.
\newblock In \bibinfo{editor}{W.~Sieg} (Ed.), {\it \bibinfo{booktitle}{Acting
  and reflecting}\/} (pp. \bibinfo{pages}{143--170}).
\newblock \bibinfo{address}{Dordrecht}: \bibinfo{publisher}{Kluwer} volume
  \bibinfo{volume}{211} of {\it \bibinfo{series}{Synthese Library}\/}.
\bibitem[{Simon(1957)}]{simon1957models}
\bibinfo{author}{Simon, H.~A.} (\bibinfo{year}{1957}).
\newblock {\it \bibinfo{title}{Models of man: social and rational; mathematical
  essays on rational human behavior in society setting}\/}.
\newblock \bibinfo{publisher}{Wiley}.
\bibitem[{Tarski(1951)}]{tarski1951decision}
\bibinfo{author}{Tarski, A.} (\bibinfo{year}{1951}).
\newblock \bibinfo{title}{A decision method for elementary algebra and
  geometry}, .
\bibitem[{Troffaes \& de~Cooman(2003)}]{troffaes2003extension}
\bibinfo{author}{Troffaes, M. C.~M.}, \& \bibinfo{author}{de~Cooman, G.}
  (\bibinfo{year}{2003}).
\newblock \bibinfo{title}{{Extension of coherent lower previsions to unbounded
  random variables}}.
\newblock {\it \bibinfo{journal}{Intelligent systems for information
  processing: from representation to applications}\/},  (pp.
  \bibinfo{pages}{277--288}).
\bibitem[{Walley(1991)}]{walley1991}
\bibinfo{author}{Walley, P.} (\bibinfo{year}{1991}).
\newblock {\it \bibinfo{title}{Statistical Reasoning with Imprecise
  Probabilities}\/}.
\newblock \bibinfo{address}{New York}: \bibinfo{publisher}{Chapman and Hall}.
\bibitem[{Walley et~al.(2004)Walley, Pelessoni \& Vicig}]{walley2004direct}
\bibinfo{author}{Walley, P.}, \bibinfo{author}{Pelessoni, R.}, \&
  \bibinfo{author}{Vicig, P.} (\bibinfo{year}{2004}).
\newblock \bibinfo{title}{Direct algorithms for checking consistency and making
  inferences from conditional probability assessments}.
\newblock {\it \bibinfo{journal}{Journal of Statistical Planning and
  Inference}\/},  {\it \bibinfo{volume}{126}\/}, \bibinfo{pages}{119--151}.
\bibitem[{Williams(1975)}]{williams1975}
\bibinfo{author}{Williams, P.~M.} (\bibinfo{year}{1975}).
\newblock {\it \bibinfo{title}{Notes on conditional previsions}\/}.
\newblock \bibinfo{type}{Technical Report} School of Mathematical and Physical
  Science \bibinfo{address}{University of Sussex, UK}.

\end{thebibliography}

\end{document}